\documentclass{amsart}

\usepackage{graphicx}%
\usepackage{multirow}%
\usepackage{amsmath,amssymb,amsfonts}%
\usepackage{amsthm}%
\usepackage{mathrsfs}%
\usepackage[title]{appendix}%
\usepackage{xcolor}%
\usepackage{hyperref}%
\usepackage[capitalise,noabbrev]{cleveref}
\usepackage{bm}
\usepackage{mathtools}

\usepackage{import}

\newtheorem{theorem}{Theorem}[section]
\newtheorem{lemma}[theorem]{Lemma}
\newtheorem{proposition}[theorem]{Proposition}
\newtheorem{cor}[theorem]{Corollary}

\theoremstyle{definition}
\newtheorem{definition}[theorem]{Definition}

\theoremstyle{remark}
\newtheorem{remark}[theorem]{Remark}

\numberwithin{equation}{section}

\begin{document}

\title{Saddle Point Configurations for Spherical Ferromagnets}

\author{Stephen Gustafson}
\address{Department of Mathematics, University of British Columbia, Vancouver, BC V6T 1Z2, Canada}
\email{gustaf@math.ubc.ca}

\author{Daniel Meinert}
\address{Lehrstuhl für Angewandte Analysis, RWTH Aachen University, Kreuzherrenstr. 2, 52062 Aachen, Germany}
\email{meinert@eddy.rwth-aachen.de}

\author{Christof Melcher}
\address{Lehrstuhl für Angewandte Analysis, RWTH Aachen University, Kreuzherrenstr. 2, 52062 Aachen, Germany}
\email{melcher@math1.rwth-aachen.de}


\date{\today}


\begin{abstract}
We investigate saddle point configurations in spherical ferromagnets with perpendicular anisotropy. These are modeled by a micromagnetic energy functional on the unit sphere that leads to the emergence of the so-called curvature induced Dzyaloshinskii--Moriya interaction. For this functional we establish the existence of two distinct types of saddle points with zero mapping degree. We use a parabolic flow approach inspired by the harmonic map heat flow where for certain highly symmetric initial conditions we can rule out finite and infinite time blowup. Hence, we obtain solutions of the underlying Euler--Lagrange equation in the long time limit. This is in contrast to harmonic maps between two-spheres, where every map is a local minimizer of the energy functional.
\end{abstract}

\maketitle

\section{Introduction and Main Results}
In this work, we study a spherical thin ferromagnet with perpendicular anisotropy governed by the energy functional
\begin{equation} \label{eq:def_energy}
    \mathcal{E}(\bm{m}) = \frac{1}{2} \int_{\mathbb{S}^2} |\nabla \bm{m}|^2  + \kappa (1- \left\langle \bm{m},\bm{\nu} \right\rangle^2) \, \mathrm{d} \sigma 
\end{equation}
$\bm{m}\colon \mathbb{S}^2\to\mathbb{S}^2$, where $\mathbb{S}^2 \subset \mathbb{R}^3$ is the unit sphere, $\bm{\nu}\in \mathbb{S}^2$ is the outer unit normal field, and $\kappa > 0$. This geometric model was first proposed in the physics literature for spherical shells by Kravchuk et al. in~\cite{kravchuk2016}, where it has been found that the curvature of the underlying manifold induces an interaction similar to that of Dzyaloshinskii--Moriya interaction (DMI), the so-called ``curvature induced DMI''. This can be demonstrated using stereographic coordinates on $\mathbb{S}^2$ and a co-rotating frame~\cite{melcher2019}. This model combines aspects of energetic, topological, and geometrical nature. In fact, this energy is not invariant under individual rotations of domain or target, which is the key feature in stabilizing skyrmionic structures.

The first term of the integral, called the Dirichlet energy, represents short-range exchange interactions and favors uniform fields. The second term, the anisotropy term, corresponds to an \emph{easy-normal anisotropy}, which energetically prefers fields aligned with the sphere's normal. These terms compete, shaping the magnetization configuration.

For (sufficiently regular) fields $\bm{m}\colon \mathbb{S}^2\to \mathbb{S}^2$ we define the mapping degree, or Skyrmion number
\begin{equation} \label{eq:def_degree}
    Q(\bm{m}) = \frac{1}{4\pi} \int_{\mathbb{S}^2}\bm{m}^* \mathrm{d} \sigma \, ,
\end{equation}
where $\bm{m}^* \mathrm{d} \sigma$ denotes the pullback of the oriented area element on $\mathbb{S}^2$. This integer-valued topological invariant separates maps into distinct sectors~\cite{lee2012}. We call critical points of $\mathcal{E}$ with mapping degree $Q(\bm{m}) = 0$ \emph{skyrmions}. The model is closely related to harmonic maps between spheres, which minimize the Dirichlet energy and are represented by rational functions~\cite{eells1978}. Harmonic maps have energy $4\pi Q(\bm{m})$ and are always local minimizers of the Dirichlet energy, i.e., no saddle points of the Dirichlet energy exist, see e.g.~\cite{brezis1985}. 

Recent studies examined $\mathcal{E}$ for local minimizers under constraints on mapping degree and symmetries. The outer normal field $\bm{\nu}$ and its inversion $-\bm{\nu}$ are critical points with $Q(\pm \bm{\nu}) = \pm 1$. Di~Fratta et al.~\cite{difratta2019} showed that for $\kappa > 4$ these are the global minimizers of $\mathcal{E}$. We consider these fields as the trivial critical points. Melcher and Sakellaris~\cite{melcher2019} proved the existence of minimizers for $Q(\bm{m}) = 0$ and all $\kappa > 0$. Schroeder~\cite{schroeder2024} showed minimizers in the class of axisymmetric maps exist for all $\kappa$ and for $\kappa >24$ these are local minimizers among all maps. The common feature of these minimizers is that they consist of the outer normal field $\bm{\nu}$ on almost the entire sphere, except for a small region around one of the poles where the field switches to the inner normal $-\bm{\nu}$. This structure is interpreted as the localized skyrmion.  

We demonstrate the existence of two distinct axisymmetric saddle points of $\mathcal{E}$:
\begin{theorem} \label{theo:saddle_point}
    There exists $\kappa_0 \geq 4$ such that for all $\kappa > \kappa_0$ there exists a critical point $\bm{m}_\kappa \in H^1(\mathbb{S}^2,\mathbb{S}^2)$ of $\mathcal{E}$ with $Q(\bm{m}_\kappa) =0$ that is a saddle point in the sense that it is neither a local minimizer nor a local maximizer of $\mathcal{E}$ in $H^1(\mathbb{S}^2,\mathbb{S}^2)$.
\end{theorem}

 We call the critical points from \cref{theo:saddle_point} saddle points of first type. We have no theoretical upper bound for $\kappa_0$ but numerical simulations suggest that $\kappa_0 < 6.7$ and, in fact, we conjecture that $\kappa_0 = 4$. Furthermore, again supported by numerical evidence, we conjecture that for $\kappa \leq 4$ the saddle point of first type transforms into the global minimizer of $\mathcal{E}$, making $\kappa = 4$ a bifurcation point. Also for $\kappa < 4$ there exist saddle points. These are distinct from the one of first type in their geometric structure (see below). We call them saddle points of second type.

\begin{theorem} \label{theo:saddle_point_2}
    There exists $0< \kappa_1 < 4$ such that for all $\kappa > \kappa_1$ there exists a critical point $\bm{m}_\kappa \in H^1(\mathbb{S}^2,\mathbb{S}^2)$ of $\mathcal{E}$ with $Q(\bm{m}_\kappa) = 0$ that is a saddle point in the above sense and that is distinct from the (possible) saddle point from \cref{theo:saddle_point}.
\end{theorem}

\begin{remark}[Symmetry of the saddle points]
    The saddle points of both types belong to a special class of maps as long as $\kappa \geq 4$, the so-called \emph{hemispheric maps}, see \cref{def:hemispheric_map}. In particular, this means that the maps are axisymmetric around a given axis and invariant under the antipodal map $\bm{m}(x) \mapsto \bm{m}(-x)$.
\end{remark}

\begin{figure}[ht]
    \begin{center}
        \includegraphics{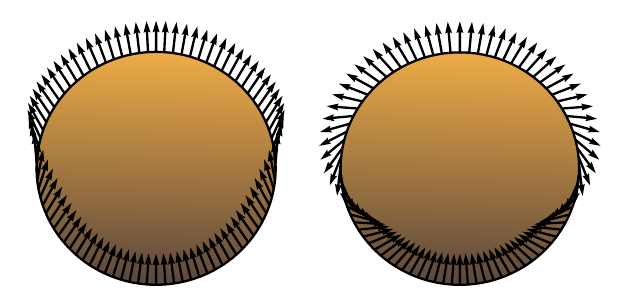}
        \caption{Great circle cross-section plot of the saddle points of first type (left) and of second type (right) for $\kappa = 10$ obtained by numerical simulations}
        \label{fig:saddle_points}
    \end{center}
\end{figure}
\cref{fig:saddle_points} shows the two saddle point types for $\kappa = 10$, with cross-sections of the sphere along a great circle. Both fields align with the outer normal $\bm{\nu}$ on the upper hemisphere and switch to the inner normal $-\bm{\nu}$ on the lower hemisphere. Due to their hemispheric symmetry, these maps do not feature any localized structures.

The structural difference between both types is that the saddle point of first type is homotopic to the constant field in the class of axisymmetric maps, while the second type undergoes a full rotation from pole to pole, disallowing such a homotopy. The saddle point of second type covers $\mathbb{S}^2$ once per hemisphere with opposite orientation, resembling a \emph{skyrmionium}, a planar structure of nested skyrmions with opposite orientations~\cite{finazzi2013,komineas2015,komineas2015a}. Thus, both maps have degree $Q(\bm{m}) = 0$, however they belong to distinct topological sectors within axisymmetric maps. Asymptotically for large $\kappa$, both saddle points become equivalent: they separate $\pm \mathrm{Id}$ on the upper and lower hemisphere with a sharp domain wall at the equator. 

To establish the existence of these critical points, we introduce a parabolic evolution equation analogous to the harmonic map heat flow. Specifically, we consider time-dependent fields $\bm{m}$ governed by the initial value problem
\begin{equation*}
    \begin{aligned}
        \bm{m}_t &=\Delta \bm{m} + \kappa\left\langle \bm{m},\bm{\nu} \right\rangle \bm{\nu} +  \left( |\nabla \bm{m}|^2 - \kappa \left\langle \bm{m},\bm{\nu} \right\rangle^2 \right) \bm{m}\\
    \bm{m}|_{t=0} &= \bm{m}_0 \, ,
    \end{aligned}
\end{equation*}
where $\bm{m}_0\colon \mathbb{S}^2 \to \mathbb{S}^2$ is a given initial map. We call this equation the \emph{skyrmionic map heat flow} (SMHF). It is a semilinear parabolic partial differential equation and is the $L^2$-gradient flow of the energy functional $\mathcal{E}$. This equation does not describe the dynamical behavior of magnetization states, as this one is governed by the Landau--Lifshitz--Gilbert equation~\cite{landau1935,gilbert2004,melcher2012}, which states
\begin{equation*}
    \bm{m}_t =\bm{m} \times \left(\Delta\bm{m} + \kappa\left\langle \bm{m},\bm{\nu} \right\rangle \bm{\nu}\right) - \lambda \bm{m} \times \bm{m} \times \left(\Delta\bm{m} + \kappa\left\langle \bm{m},\bm{\nu} \right\rangle \bm{\nu}\right) \, ,
\end{equation*}
where $\lambda > 0$ is the so-called Gilbert damping parameter and $\times$ denotes the vector cross product in $\mathbb{R}^3$. Noticing that 
\begin{equation*}
    -\bm{m} \times \bm{m} \times \left(\Delta\bm{m} + \kappa\left\langle \bm{m},\bm{\nu} \right\rangle \bm{\nu}\right) = \Delta \bm{m} + \kappa\left\langle \bm{m},\bm{\nu} \right\rangle \bm{\nu} + \bm{m} \left( |\nabla \bm{m}|^2 - \kappa \left\langle \bm{m},\bm{\nu} \right\rangle^2 \right) \, ,
\end{equation*}
we see that the SMHF describes the dynamics with only the damping term present. 

The heat flow serves as an efficient method to study $\mathcal{E}$, as stationary solutions are critical points of $\mathcal{E}$. The approach here is analogous to the introduction of the harmonic map heat flow (HMHF) by Eells and Sampson in~\cite{eells1964} in order to find harmonic representatives of homotopy classes of maps between Riemannian manifolds. For their seminal result, they required the domain manifold to be of negative sectional curvature. Struwe expanded the theory of the HMHF for surfaces in~\cite{struwe1985}, replacing the curvature condition by a topological one. In essence, he showed that the HMHF is smooth up to singular blowup points in space-time. At these points, harmonic 2-spheres separate and take away a multiple of $4\pi$ of energy. In fact, since the functional $\mathcal{E}$ only differs from the Dirichlet energy by a term of lower order, we can transfer many results from the study of the HMHF to the SMHF (see~\cite{struwe2008,lin2008,helein2002} and references therein for details on the HMHF). The main one is the analogous existence result of Struwe, with the only difference that the limit map is a critical point of $\mathcal{E}$. 

With this existence result in hand, the proofs of \cref{theo:saddle_point,theo:saddle_point_2} require two further steps. First, we show that there exist initial maps $\bm{m}_0$ such that the solution to the heat flow equation cannot blow up at any point in time, $T=\infty$ included. This step is inspired by a global existence result of Chang and Ding for the HMHF~\cite{chang1991}. Then the limit function $\bm{m}_\infty$ for $t \to \infty$ of the flow is a critical point of $\mathcal{E}$ with same degree as $\bm{m}_0$. In the final step, we show that $\bm{m}_\infty$ is indeed a saddle point. The crucial element here is that $\bm{m}_\infty$ obeys the hemispheric symmetry. By breaking this symmetry manually, we can construct maps with lower energy than $\bm{m}_\infty$, which allows us to establish the saddle point property.

\section{Preliminaries and Notation}
\label{chapter:preliminaries}
\subsection{General setting}
For $k\in \mathbb{N}_0$ we define the Sobolev spaces of $\mathbb{S}^2$-valued maps as
\begin{equation*}
    H^{k}(\mathbb{S}^2,\mathbb{S}^2) = \left\{ \bm{m}\in H^{k}(\mathbb{S}^2,\mathbb{R}^3) : \bm{m}(x) \in \mathbb{S}^2 \text{~for~almost~all~} x \in \mathbb{S}^2 \right\} \, ,
\end{equation*}
where  $H^{k}(\mathbb{S}^2,\mathbb{R}^3)$ denotes the usual Sobolev space of maps from $\mathbb{S}^2$ to $\mathbb{R}^3$ with square integrable derivative up to $k$-th order, with the convention $H^0(\mathbb{S}^2,\mathbb{R}^3) = L^2(\mathbb{S}^2,\mathbb{R}^3)$. For details on Sobolev spaces modeled on Riemannian surfaces we refer to~\cite{hebey1996,eichhorn2007}. The functional $\mathcal{E}$ defined in \eqref{eq:def_energy} then becomes a map $\mathcal{E}\colon H^1(\mathbb{S}^2,\mathbb{S}^2) \to \mathbb{R}$. For the map $Q$ defined in \eqref{eq:def_degree} we furthermore get by density of $C^\infty (\mathbb{S}^2,\mathbb{S}^2)$ in $H^1 (\mathbb{S}^2,\mathbb{S}^2)$ that this map extends continuously to a map $Q\colon H^1 (\mathbb{S}^2,\mathbb{S}^2) \to \mathbb{Z}$~\cite{schoen1983}.

Critical points of $\mathcal{E}$ are maps $\bm{m}\in H^1(\mathbb{S}^2,\mathbb{S}^2)$ such that if we define for $\bm{\varphi} \in H^1 \cap L^\infty(\mathbb{S}^2,\mathbb{R}^3)$ and $\varepsilon< \|\bm{\varphi} \|_{L^\infty}$ the map 
\begin{equation*}
    \bm{m}_\varepsilon = \frac{\bm{m} + \varepsilon \bm{\varphi}}{|\bm{m} + \varepsilon \bm{\varphi}|} \, ,
\end{equation*}
then $\bm{m}_\varepsilon$ satisfies
\begin{equation*}
    \frac{\mathrm{d}}{\mathrm{d} \varepsilon} \mathcal{E}(\bm{m}_\varepsilon)\Big|_{\varepsilon=0} = 0 \, .
\end{equation*}
This is equivalent to requiring $\bm{m}$ to solve the following equation in a weak sense in $H^1(\mathbb{S}^2,\mathbb{R}^3)$ with test functions in the space $H^1\cap L^\infty(\mathbb{S}^2,\mathbb{R}^3)$
\begin{equation} \label{eq:EL}
    0 =\Delta \bm{m} + \bm{m} |\nabla \bm{m}|^2 + \kappa\left\langle \bm{m},\bm{\nu} \right\rangle \Pi_{\bm{m}}\bm{\nu} \, ,  
\end{equation}
where $\Pi_{\bm{v}}$ is the pointwise orthogonal projection onto the tangent space of $\mathbb{S}^2$ at $\bm{v}$ embedded into $\mathbb{R}^3$ defined as
\begin{equation*}
    \Pi_{\bm{v}}\bm{w} = \bm{w} - \left\langle \bm{v},\bm{w} \right\rangle \bm{v} \, 
\end{equation*}
for $\bm{v} \in \mathbb{S}^2$ and $\bm{w} \in \mathbb{R}^3$. We call \eqref{eq:EL} the Euler--Lagrange equation of $\mathcal{E}$. By a result of Hélein~\cite{helein2002}, we get that weak solutions are smooth maps, see~\cite{schroeder2024}. Moreover, we can transfer another result of the theory of harmonic maps to our setting, first proven by Sacks and Uhlenbeck~\cite{sacks1981}, namely that solutions of \eqref{eq:EL} on pointed domains can be smoothly extended. The proof is completely analogous to the case of harmonic maps as the lower order anisotropy term does not affect the arguments.

\begin{proposition}[Removability of point singularities] \label{cor:point_removability}
    Let $\bm{m} \in H^1(\mathbb{S}^2,\mathbb{S}^2)$ and let $\{z_1, \ldots, z_k\}$ be a finite set such that $\bm{m}$ is a weak solution of \eqref{eq:EL} on $\mathbb{S}^2\setminus\{z_1, \ldots, z_k\}$. Then $\bm{m}$ extends to a smooth map on $\mathbb{S}^2$ that is a solution of \eqref{eq:EL}.
\end{proposition} 

\subsection{Notation on the sphere}
We use two types of scripts for elements on the sphere $\mathbb{S}^2$. As we consider maps $\mathbb{S}^2 \to \mathbb{S}^2$, we differentiate elements in domain and target. We write elements in the domain $\mathbb{S}^2$ in normal script letters, e.g., $x \in \mathbb{S}^2$. In the view of $\mathbb{S}^2$ as a Riemannian manifold, the expression $x$ in coordinates depends on the chosen parametrization. For elements in the target, we use boldface letters, e.g., $\bm{m} \in \mathbb{S}^2$. In this case, we view $\mathbb{S}^2$ as a subset of $\mathbb{R}^3$ and boldface letters denote column vectors in $\mathbb{R}^3$.

Sometimes we express maps defined on the sphere in ``stereographic coordinates centered at $x_0$'' for some point $x_0 \in \mathbb{S}^2$. This means that we consider the map in the parametrization $\bm{\Phi}_{x_0}\colon \mathbb{R}^2 \to \mathbb{S}^2$ given by
\begin{equation*}
    (x_1, x_2) \mapsto Q_{x_0}\left[\frac{2}{1+|x|^2} \left(x_1, x_2, \frac{1-|x|^2}{2}\right)^T \right]\, ,
\end{equation*}
where $Q_{x_0}\in SO(3)$ is the rotation matrix that maps $\hat{\bm{e}}_3$ to $x_0$. In particular, this means $\bm{\Phi}_{x_0}(0,0) = x_0$.

\subsection{Axisymmetric maps}

We introduce joint rotations or reflections in target and domain. For $R \in O(3)$ we define $\bm{m} \mapsto \bm{m}_R$, where
\begin{equation} \label{eq:definition_O3_action}
    \bm{m}_R (x) = R^{-1}\bm{m}(R.x) \, .
\end{equation}
Here, $R.x$ is the action of the transform $R$ on the point $x\in \mathbb{S}^2$ and the outer action is to be interpreted as the matrix vector product in $\mathbb{R}^3$. These joint rotations or reflections define a right group action of $O(3)$ on the space of functions $\mathbb{S}^2 \to \mathbb{S}^2$. Visually, $\bm{m}_R$ can be interpreted as the field after applying the action of $R^{-1}$ to the sphere with the vector field $\bm{m}$ attached to it at its base points, as depicted in \cref{fig:rotation}.

\begin{figure}[t]
    \centering
        \includegraphics{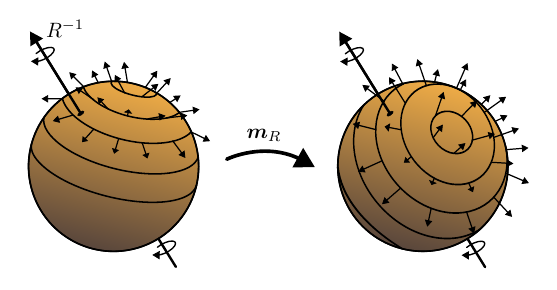} 
        \caption{The action $\bm{m}_R$ can be interpreted as the sphere with the vector field $\bm{m}$ attached to it at its base points, rotated/inverted by $R^{-1}$.} 
        \label{fig:rotation} 
\end{figure} 
    
Using the orthogonality of $R \in O(3)$, we obtain that $|\nabla \bm{m}_R| (x) = |\nabla \bm{m}|(R.x)$ and $|\bm{m}_R|(x) = |\bm{m}|(x)$ for all $x\in \mathbb{S}^2$, from which follows that this action defines an isometry on $H^1(\mathbb{S}^2,\mathbb{S}^2)$. Moreover, we have $R\bm{\nu}(x) = \bm{\nu}(R.x)$, such that we conclude that the energy $\mathcal{E}$ is invariant under the action, i.e., $\mathcal{E}(\bm{m}_R) = \mathcal{E}(\bm{m})$.  With this observation, one can also prove that for any critical point $\bm{m}\colon \mathbb{S}^2\to\mathbb{S}^2$ the map $\bm{m}_R$ is also a critical point of $\mathcal{E}$ for all $R\in O(3)$.

\begin{definition} \label{def:axisymmetric_map}
    We call a map $\bm{m}\colon\mathbb{S}^2 \to \mathbb{S}^2$ axisymmetric if and only if $\bm{m}_R = \bm{m}$ for all $R\in O(3)_{\hat{\bm{e}}_3}$, the stabilizer subgroup of $O(3)$ with respect to $\hat{\bm{e}}_3$, that is rotations around the $\hat{\bm{e}}_3$-axis and reflections in the $\hat{\bm{e}}_1\hat{\bm{e}}_2$-subspace. 
\end{definition}
Maps that are only invariant under the action of $SO(3)_{\hat{\bm{e}}_3}$, i.e., joint rotations, are called \emph{equivariant maps} and are not the same as axisymmetric maps. For instance, the map $\bm{m}(x) \equiv \hat{\bm{e}}_\varphi$, the azimuthal unit vector field, is equivariant but not axisymmetric.

We denote spherical coordinates on the sphere with azimuthal angle around the $\hat{\bm{e}}_3$-axis by 
\begin{align*}
    \bm{\hat{\Psi}}&\colon (0,\pi)\times (0,2\pi) \to \mathbb{S}^2 \, \\
    (\theta,\varphi)& \mapsto (\cos\varphi \sin \theta, \sin \varphi \sin \theta, \cos \theta)^T \, 
\end{align*}
and by $\bm{\Psi}$ the map $\bm{\hat{\Psi}}$ with the domain extended to $\mathbb{R} \times [0,2\pi]$. The following lemma tells us that all information of an axisymmetric map $\bm{m}$ is contained in a so-called profile function $h$ that depends only on the polar angle. The proof can be found in~\cite{schroeder2024}.
\begin{lemma}[{\cite[Lemma 2.5]{schroeder2024}}]
    A continuous map $\bm{m}\colon\mathbb{S}^2\to\mathbb{S}^2$ is axisymmetric if and only if there exists a continuous map $h\colon[0,\pi]\to \mathbb{R}$ such that 
    \begin{equation} \label{eq:axisymmetric_form}
        \bm{m}(\theta,\varphi) = \bm{\Psi}(h(\theta),\varphi) \quad\text{~for~all~}\quad (\theta,\varphi) \in (0,\pi) \times (0,2\pi) \, .
    \end{equation}
    We then call $h$ the \emph{profile} of $\bm{m}$. Note that by the symmetry condition $\bm{m}(\pm \hat{\bm{e}}_3) = \pm \hat{\bm{e}}_3$ we can continuously extend $h$ to $0$ and $\pi$ by $h(0) = m\pi$, $h(\pi) = n\pi$ with $m,n\in \mathbb{Z}$.
\end{lemma}

We define for $k\in \mathbb{N}$ and $0< \alpha <1$
\begin{equation*}
    \mathcal{F}^{k,\alpha} \coloneqq \{h \in  C^{k,\alpha}([0, \pi]) : h(0) = m\pi,\, h(\pi) = n\pi, \;\; m,n\in \mathbb{Z} \}\, ,
\end{equation*}
and likewise
\begin{equation*}
    \mathcal{F}^{\infty} \coloneqq \{h \in  C^{\infty}([0, \pi]) : h(0) = m\pi,\, h(\pi) = n\pi, \;\;m,n\in \mathbb{Z} \}\, ,
\end{equation*}
and finally
\begin{equation*}
    \mathcal{PF}^1 \coloneqq \{h \in  PC^1([0, \pi]) : h(0) = m\pi,\, h(\pi) = n\pi,\;\; m,n\in \mathbb{Z} \}\, ,
\end{equation*}
where $PC^1([0,\pi])$ is the space of continuous and piecewise continuously differentiable functions on $[0,\pi]$.

We state the following regularity results for axisymmetric maps and their profile functions without proof. It is straightforward and only requires a careful analysis of the derivatives at the poles. This can be done by using stereographic coordinates centered at these points and checking differentiability there by hand. 
\begin{lemma} \label{lemma:regularity_m_h}
    Let $\bm{m} \colon\mathbb{S}^2 \to \mathbb{S}^2$ be axisymmetric with profile $h$. Then for all $0<\alpha<1$ we have $\bm{m} \in  C^{1,\alpha}(\mathbb{S}^2,\mathbb{S}^2)$ if $h \in \mathcal{F}^{1,\alpha}$. Conversely, for all $k\in \mathbb{N}$ and $0<\alpha<1$ we get $h \in \mathcal{F}^{k,\alpha}$ if $\bm{m} \in C^{k,\alpha}(\mathbb{S}^2,\mathbb{S}^2)$. Moreover, if $h \in \mathcal{PF}^1$, then $\bm{m} \in H^1(\mathbb{S}^2,\mathbb{S}^2)$.
    \begin{remark}
        The weaker first statement results from the fact that higher order derivatives of $\bm{m}$ at the poles in $\varphi$-direction cannot necessarily be bounded without more restrictions on $h$. 
    \end{remark}
\end{lemma}

For axisymmetric maps expressed in spherical coordinates the energy reduces to the following expression
\begin{equation} \label{eq:reduced_energy}
    \mathcal{E}(\bm{m})=\pi \int_0^\pi {h'}^2 \sin \theta+\frac{\sin ^2 h}{\sin \theta}+\kappa \sin ^2(h-\theta) \sin \theta \, \mathrm{d} \theta\eqqcolon 2\pi E(h) \, ,
\end{equation}
which follows from inserting \eqref{eq:axisymmetric_form} into $\mathcal{E}$ and computing the individual terms. Similarly, for the mapping degree of an axisymmetric map we compute, using the explicit formula for spherical coordinates
\begin{equation*}
    \bm{m}^*\mathrm{d} \sigma  = \left\langle \bm{m},\frac{\partial \bm{m}}{\partial \theta} \times \frac{\partial \bm{m}}{\partial \varphi} \right\rangle  \mathrm{d} \theta \mathrm{d} \varphi \, ,
\end{equation*}
the expression
\begin{equation}\label{eq:Q_axisymmetric}
    Q(\bm{m}) =\frac{1}{4 \pi} \int_{\mathbb{S}^2} \bm{m}^* \mathrm{d} \sigma =\frac{1}{2} \int_0^\pi h' \sin h \, \mathrm{d} \theta=\frac{1}{2} \left(\cos (h(0))-\cos (h(\pi))\right) \, . 
\end{equation}

The following result shows the usefulness of the axisymmetric ansatz: we reduce the problem to a one dimensional one, which has already been observed in~\cite{schroeder2024} and is also a known result for harmonic maps, see e.g.,~\cite{chang1991}.
\begin{lemma} \label{lemma:EL_h}
    Let $\bm{m}$ be axisymmetric with profile $h$. Then $\bm{m}$ solves the Euler--Lagrange equation \eqref{eq:EL} if and only if $h$ solves
    \begin{equation} \label{eq:EL_h}
        0=h''+\frac{\cos \theta}{\sin \theta} h'-\frac{\sin 2 h}{2 \sin ^2 \theta} - \frac{\kappa}{2} \sin(2h-2\theta)
    \end{equation}
    on $(0,\pi)$.
    \begin{proof}
        Computing and inserting the individual terms in the Euler--Lagrange equation with the axisymmetric ansatz in spherical coordinates we find
        \begin{align*}
            &\phantom{{}={}}\; \Delta \bm{m} + |\nabla \bm{m}|^2 \bm{m} - \kappa\left\langle \bm{m},\bm{\nu} \right\rangle \Pi_{\bm{m}} \bm{\nu} \\            
            &= \left(h'' + \frac{\cos\theta}{\sin\theta} h' - \frac{\sin h \cos h}{\sin^2\theta}-\frac{\kappa}{2}\sin(2h-2\theta)\right) \begin{pmatrix} \cos\varphi \cos h\\ \sin \varphi\cos h \\ -\sin h\end{pmatrix} \, .
        \end{align*}
        Therefore, the left-hand side is zero for all $(\theta,\varphi) \in (0,\pi)\times [0,2\pi)$ if and only if the coefficient of the vector on the right-hand side is zero, which is equivalent to \eqref{eq:EL_h}. Hence, if $\bm{m}$ solves \eqref{eq:EL} on $\mathbb{S}^2$, then $h$ solves \eqref{eq:EL_h} on $(0,\pi)$. Conversely, if $h$ solves \eqref{eq:EL_h} on $(0,\pi)$, then $\bm{m}$ solves \eqref{eq:EL} on $\mathbb{S}^2\setminus \{\pm \hat{\bm{e}}_3\}$ and by \cref{cor:point_removability} on the whole sphere. 
    \end{proof}
\end{lemma}
In combination with \cref{lemma:regularity_m_h} we conclude 
\begin{cor}
    Let $h \in \mathcal{PF}^1$ be a solution to \eqref{eq:EL_h}. Then $h \in \mathcal{F}^\infty$.
\end{cor}

Clearly, \eqref{eq:EL_h} is in fact the Euler--Lagrange equation of the reduced energy functional $E$. Computing the first variation of $E$ at $h$ for a function $g \in C^1([0,\pi])$ with $g(0) = g(\pi) = 0$, we find
\begin{align*}
    \delta E[h](g) &= \frac{\mathrm{d}}{\mathrm{d} t} E(h + tg) \Big|_{t=0} \\
    &= \int_0^\pi h' g' \sin \theta + \frac{\sin h \cos h}{\sin \theta} g + \kappa \sin(h-\theta) \cos(h-\theta) g \sin\theta \, \mathrm{d} \theta \\
    &= \int_0^\pi \left( -h'' - \frac{\cos\theta}{\sin\theta} \, h' + \frac{\sin 2h}{2\sin^2\theta} + \frac{\kappa}{2} \sin(2h-2\theta)\right) g \sin\theta \, \mathrm{d} \theta \, . 
\end{align*}
For the second variation at $h$ in direction $g$ we gain
\begin{align*}
    \delta^2 E[h](g) &= \frac{\mathrm{d}^2}{\mathrm{d} t^2} E(h + tg) \Big|_{t=0} \\ 
    &= \int_0^\pi {g'}^2 \sin \theta+\left(\frac{\cos2 h}{\sin^2 \theta} +\kappa\cos(2h-2\theta)  \right) g^2 \sin \theta \, \mathrm{d} \theta \, . \addtocounter{equation}{1}\tag{\theequation} \label{eq:second_variation_E}
\end{align*}
Let $h$ be a solution of \eqref{eq:EL_h}, and for $\varepsilon \in \mathbb{R}$ let $h_\varepsilon = h + \varepsilon g$ where $g\in C^\infty([0,\pi])$ with $g(0) = g(\pi) = 0$. By Taylor's theorem we then obtain 
\begin{equation}\label{eq:energy_difference_hg_eps}
    E(h_\varepsilon) - E(h) = \frac{\varepsilon^2}{2} \, \delta^2E[h](g) + o(\varepsilon^2)\, ,
\end{equation}
which by \eqref{eq:reduced_energy} directly transfers to the energy difference of the respective axisymmetric maps $\bm{m}_\varepsilon$ and $\bm{m}$.

\subsection{An exact solution}
As the normal field $\bm{\nu}$ is a solution to \eqref{eq:EL} for any value of $\kappa$, its profile $h(\theta) = \theta$ is a solution to \eqref{eq:EL_h}, which can easily be verified. For a special value of $\kappa$, there exists another analytic solution.

\begin{proposition} \label{prop:kappa_4_solution}
    Let $\kappa = 4$. Then $h(\theta) = 2 \theta$ is a solution to \eqref{eq:EL_h}.
    \begin{proof}
        We insert the map into the right-hand side of \eqref{eq:EL_h} with $h'' = 0$ and $h' = 2$ to find by means of trigonometric identities
        \begin{equation*}
            h''+\frac{\cos \theta}{\sin \theta} h'-\frac{\sin 2 h}{2 \sin ^2 \theta} - \frac{\kappa}{2} \sin(2h-2\theta) = 2\,\frac{\cos \theta}{\sin \theta} - \frac{\sin 4\theta}{2 \sin^2 \theta} - \frac{\kappa}{2} \sin 2\theta  = \left( 2 - \frac{\kappa}{2} \right) \sin 2\theta \, .
        \end{equation*}
        Therefore, if $\kappa = 4$, then $h(\theta) = 2\theta$ is a solution to \eqref{eq:EL_h}.
    \end{proof}
\end{proposition}
In \cref{chapter:saddlepoint} we will see that the function $\bm{m}$ corresponding to this profile is in fact a saddle point of the energy functional $\mathcal{E}$. 

\subsection{Skyrmionic map heat flow}

In order to find solutions for \eqref{eq:EL}, we study the \emph{skyrmionic map heat flow} (SMHF), i.e.\  for a given function $\bm{m}_0\colon \mathbb{S}^2\to\mathbb{S}^2$ we consider the parabolic problem
\begin{equation} \label{eq:SHE}
    \begin{aligned}
        \bm{m}_t &=\Delta \bm{m} + \kappa\left\langle \bm{m},\bm{\nu} \right\rangle \bm{\nu} + \bm{m} \left( |\nabla \bm{m}|^2 - \kappa \left\langle \bm{m},\bm{\nu} \right\rangle^2 \right) \\
    \bm{m}|_{t=0} &= \bm{m}_0 \, , 
    \end{aligned}
\end{equation}
where $\bm{m}_t$ denotes the time derivative of $\bm{m}$. Stationary solutions of this evolution problem correspond to solutions of \eqref{eq:EL}. As mentioned in the introduction, the flow is closely related to the harmonic map heat flow (HMHF) and the anisotropy does not change much the situation. Thus, we can adapt Struwe's HMHF result to our case.

\begin{theorem} \label{theo:flow_existence}
    For any $\kappa \geq 0$ and initial data $\bm{m}_0 \in H^{1}(\mathbb{S}^2, \mathbb{S}^2)$ there exists a weak solution $\bm{m}\colon \mathbb{S}^2 \times (0,\infty) \to \mathbb{S}^2$ of \eqref{eq:SHE} that is smooth on $\mathbb{S}^2 \times (0,\infty)$ away from at most finitely many points $\left(\bar{x}_k, \bar{t}_k\right)\in \mathbb{S}^2\times (0,\infty)$, $0 \leq k \leq \bar{K}$, that satisfies the energy inequality $\mathcal{E}(\bm{m}(t)) \leq \mathcal{E}(\bm{m}(s))$ for all $0 \leq s \leq t$, and that attains its initial map continuously in $H^{1}(\mathbb{S}^2,\mathbb{S}^2)$. The solution $\bm{m}$ is unique in this class.

    At a singularity $(\bar{x}, \bar{t})$, smooth harmonic maps separate in the following sense. There exist sequences $t_j \nearrow \overline{t}$, $(a_j) \subset \mathbb{R}^2$ with $\lim_{j \rightarrow \infty} a_j = 0$, $(\lambda_j) \subset \mathbb{R}_+$ with $\lim_{j \rightarrow \infty} \lambda_j = 0$, and a harmonic map $\bm{\omega}\colon \mathbb{S}^2 \to \mathbb{S}^2$ such that in stereographic coordinates centered at $\overline{x}$ we have
    \begin{equation} \label{eq:blowup_convergence}
        \bm{m}(\lambda_j x + a_j, t_j) \longrightarrow \bm{\omega}(x) \, \, , \, \text{~as~} \,  j \to \infty \, ,
    \end{equation}
    in $H_\mathrm{loc}^1(\mathbb{R}^2,\mathbb{R}^3)$ and $L_\mathrm{loc}^\infty(\mathbb{R}^2,\mathbb{R}^3)$.

    For the asymptotic behavior, as $t_j \rightarrow \infty$ suitably, the sequence of maps $\bm{m}(t_j)$ converges weakly in $H^{1}(\mathbb{S}^2, \mathbb{S}^2)$ to a smooth solution $\bm{m}_{\infty}\colon\mathbb{S}^2 \rightarrow \mathbb{S}^2$ of \eqref{eq:EL}, and smoothly away from finitely many points $\tilde{x}_k \in  \mathbb{S}^2, 0 \leq k \leq \tilde{K}$, where again harmonic spheres separate in the same sense as above.
    \end{theorem}
The proof of this theorem can be carried out analogously to the proof of the corresponding theorem for the harmonic map heat flow, e.g. as in~\cite[Thm.\ 6.6]{struwe2008}. The main difference is that we have to take care of the additional anisotropy term in the energy functional $\mathcal{E}$. However, this term is of lower order and does not affect the crucial estimates and properties needed such as local energy growth estimates, local $\varepsilon$-regularity or the asymptotic behavior of the flow. In fact, whenever we have weak $H^1$-convergence of the flow, by the Rellich--Kondrachov theorem we immediately obtain strong convergence of the additional terms as they are of type $L^p$ for $p \geq 2$. For full details, we refer to~\cite{meinert2025}. Related results in similar settings can be found in~\cite{harpes2004,doring2017,goldys2025}.

In fact, it is also possible to improve the statement of the theorem and include a finer blowup analysis with the inclusion of blowup trees, as in~\cite{qing1995,lin2002,wang2017}. For our purposes, however, the convergence as described is sufficient.

\subsection{SMHF for axisymmetric maps}

For axisymmetric initial data, the symmetry is sustained under the SMHF. This observation allows us to reduce the problem to a one-dimensional parabolic equation on the interval $[0,\pi]$.

\begin{proposition} \label{prop:axisymmetric_solution_preserved}
    Let $\bm{m}_0 \in H^1(\mathbb{S}^2,\mathbb{S}^2)$ be axisymmetric and let $h_0\colon[0, \pi] \rightarrow \mathbb{R}$ be its profile with $h_0(0)= m\pi,\, h_0(\pi)=n \pi$ for some $m,n\in \mathbb{Z}$.

    Then the solution $\bm{m}$ to \eqref{eq:SHE} with initial data $\bm{m}_0$ provided by Theorem \ref{theo:flow_existence} is also axisymmetric for all $t\in(0,T)$, where $T>0$ is the maximal existence time until the first blowup, if present. In particular, there exists a map $h:C^\infty([0,\pi] \times (0,T))$ such that 
    \begin{equation} \label{eq:axisymmetric_solution_preserved_axi}
        \bm{m}(\theta,\varphi, t) = \bm{\Psi}(h(\theta, t), \varphi) \text{~for~all~} (\theta, \varphi, t) \in [0,\pi] \times [0,2\pi) \times (0,T)\, .
    \end{equation}
    This map satisfies 
    \begin{equation}\label{eq:SHE_h}
        \begin{aligned} 
            &h_t =h_{\theta \theta}+\frac{\cos \theta}{\sin \theta} h_\theta-\frac{\sin 2 h}{2 \sin ^2 \theta}- \frac{\kappa}{2} \sin(2h-2\theta), \\
            & \begin{aligned}
                h(\theta, 0) =h_0(\theta) \quad \qquad \qquad&\text{~for~all~} \theta \in [0,\pi] ,\\
            h(0, t) = m\pi, \quad h(\pi, t)=n \pi \qquad  &\text{~for~all~} t\in (0,T) \,  .
            \end{aligned}
       \end{aligned}
    \end{equation}       
    \begin{proof} 
        Let $R \in O(3)_{\hat{\bm{e}}_3}$ be arbitrary. Then, $\bm{m}_R$ also solves \eqref{eq:SHE}. To prove this, we investigate how the individual terms in the equation transform. First, we use the identity for orthogonal transformations  
        \begin{equation*}
            \Delta_{\mathbb{S}^2} (\bm{m} \circ R) = (\Delta_{\mathbb{S}^2} \bm{m}) \circ R \, ,
        \end{equation*}
        which can be shown in stereographic coordinates using the analogous identity for maps on $\mathbb{R}^2$. Moreover, the following identities hold
        \begin{align*}
            |\nabla \bm{m}_R(x) |^2 = |\nabla \bm{m}|^2 (R.x)\, , \quad \bm{\nu}(R.x) = R \bm{\nu}(x) \;\; \text{~and~} \;\; \left\langle \bm{m}_R, \bm{\nu} \right\rangle (x) = \left\langle \bm{m}, \bm{\nu} \right\rangle (R.x) \, .
        \end{align*}
        We arrive at the following:
        \begin{align*}
            &\left(\Delta \bm{m}_R + \kappa\left\langle \bm{m}_R,\bm{\nu} \right\rangle \bm{\nu} + \bm{m}_R \left( |\nabla \bm{m}_R|^2 - \kappa \left\langle \bm{m}_R,\bm{\nu} \right\rangle^2 \right)\right)(x,t)  \\ 
            =& \, R^{-1} \left(\Delta \bm{m} + \kappa\left\langle \bm{m},\bm{\nu} \right\rangle \bm{\nu} + \bm{m} \left( |\nabla \bm{m}|^2 - \kappa \left\langle \bm{m},\bm{\nu} \right\rangle^2 \right) \right) (R.x,t) \\
            =& \, R^{-1}\,  \partial_t \bm{m} (R.x,t) = \partial_t \bm{m}_R(x,t) \, .
        \end{align*}
        Hence, $\bm{m}_R$ solves \eqref{eq:SHE} with initial data
        \begin{equation*}
            \bm{m}_R(x,0) = R^{-1} \bm{m}(R.x, 0) = R^{-1} \bm{m}_0(R.x) = \bm{m}_0(x) \, ,
        \end{equation*}
        since $\bm{m}_0$ is axisymmetric by assumption. As a result, by the uniqueness of the solution it follows that $\bm{m}_R(t) = \bm{m}(t)$ for all $t\in(0,T)$. As $R\in O(3)_{\hat{\bm{e}}_3}$ was arbitrary, we deduce that $\bm{m}(t)$ is axisymmetric. Therefore, for every $t \in (0,T)$ we get a map $h(\theta, t)$ such that \eqref{eq:axisymmetric_solution_preserved_axi} holds. Smoothness of $h$ follows from \cref{lemma:regularity_m_h} and the smoothness of $\bm{m}$. Inserting the representation \eqref{eq:axisymmetric_solution_preserved_axi} of $\bm{m}$ into \eqref{eq:SHE} gives us \eqref{eq:SHE_h} as in the proof of \cref{lemma:EL_h}. The fixed boundary values $h(0,t) = m\pi$ and $h(\pi,t) = n\pi$ follow from the axisymmetry of $\bm{m}$ and its continuity.        
    \end{proof}
\end{proposition} 

We have established that the flow preserves the boundary values of the profile. This splits the function space into finer topological classes than the mapping degree. For $m,n\in \mathbb{Z}$ we define these as
\begin{equation*}
    E_{m,n} = \left\{ h\in \mathcal{PF}^1 : h(0) = m\pi, h(\pi) = n\pi \right\}  \, . \qedhere
\end{equation*}
Moreover, by \cref{lemma:regularity_m_h}, we deduce that the axisymmetric map $\bm{m}_0$ defined by the profile $h_0$ lies in $H^1(\mathbb{S}^2, \mathbb{S}^2)$. Therefore, by the previous proposition, we conclude that for given $h_0 \in  E_{m,n}$ there exists a smooth solution $h$ to the initial boundary value problem \eqref{eq:SHE_h} with initial condition $h_0$.

In the case of a blowup, the boundary values of the profiles change by an integer multiple of $\pi$. This is the consequence of the following lemma, which is an adaption of the results in~\cite{bertsch2011}. Before we state it, we need to define the profile function in the case of axisymmetric harmonic maps between two-spheres. These are of the form
\begin{equation*}
    \bm{u}(\theta,\varphi) = \bm{\Psi}( \beta(\theta),\varphi) \, ,
\end{equation*}
where 
\begin{equation*}
    \beta(\theta) = 2\arctan\left(a\tan \left( \frac{\theta}{2} \right) \right) + b\pi  \, 
\end{equation*}
with $a \in \mathbb{R}$ and $b\in \{0,1\}$. 
\begin{lemma} \label{lemma:blow_up_pole}
    Let $\bm{m}_0 \in H^1(\mathbb{S}^2,\mathbb{S}^2)$ be axisymmetric and let $\bm{m}\colon\mathbb{S}^2 \times (0,\infty) \to \mathbb{S}^2$ be the solution to \eqref{eq:SHE} with initial data $\bm{m}_0$ provided by Theorem \ref{theo:flow_existence}. If $\bm{m}$ blows up at $(x_0,T)\in \mathbb{S}^2\times (0,\infty]$ in the sense of \eqref{eq:blowup_convergence}, then $x_0 = \pm \hat{e}_3$. 
    
    Furthermore, we can choose the sequences $(\lambda_j) \subset \mathbb{R}_+$ and $(a_j) \subset \mathbb{R}^2$ such that the harmonic map $\bm{\omega}$ separating at $x_0$ along the sequence $T_j \nearrow T$ as in \eqref{eq:blowup_convergence} is axisymmetric with nonconstant profile $\beta$. In addition, for the function $h \in E_{m,n}$ satisfying \eqref{eq:axisymmetric_solution_preserved_axi} there exists $k\in \mathbb{Z}$ such that for every $0<\theta_0<\pi$
    \begin{equation*}
        h\left(2\arctan\left(\lambda_j\tan\left(\frac{\theta}{2}\right)\right) , T_j\right) \longrightarrow \beta(\theta) + k\pi \; \, \, , \, \text{~as~} \,  j \to \infty \, ,
    \end{equation*}
    uniformly for all $\theta \in [0,\theta_0]$. For blowup at the south pole, a similar result holds by a change of coordinates.

    \begin{proof}
        By the axisymmetry it follows that $R.x_0$ is also a singular point for all $R\in O(3)_{\hat{\bm{e}}_3}$. Moreover, as $O(3)_{\hat{\bm{e}}_3}$ is a Lie-group, by the finiteness of the singular set, only $x_0 = \pm \hat{e}_3$ is possible. 

        We treat the cases of finite time blowup and blowup at infinity simultaneously and assume without loss of generality that $x_0 = \hat{e}_3$ is the only blowup point. In both cases, the sequences $(\lambda_j)$ and $(a_j)$ are chosen such that we have the energy estimate
        \begin{equation} \label{eq:proof_axisymmetric_blowup_1}
            \int_{B_{\lambda_j}(a_j)} |\nabla \bm{m}(T_j)|^2 \, \mathrm{d}x > \varepsilon > 0\, 
        \end{equation}
        for all $j\in \mathbb{N}$ for some $\varepsilon > 0$ fixed (cf.\ the proofs of Proposition 1.2 in~\cite{qing1995} and of Theorem 1.2 in~\cite{wang2017}). Moreover, by assumption, we have the convergence
        \begin{equation} \label{eq:proof_axisymmetric_blowup_2}
            \bm{m}(\lambda_j x + a_j, T_j) \longrightarrow \bm{\omega}(x) 
        \end{equation}
        in $L_\mathrm{loc}^\infty(\mathbb{R}^2, \mathbb{R}^3)$ and $H^1_\mathrm{loc}(\mathbb{R}^2,\mathbb{R}^3)$, where both maps are their representations in stereographic coordinates centered at the north pole.

        We claim that the sequence $a_j / \lambda_j$ is bounded in $\mathbb{R}^2$. By contradiction, assume there is a subsequence such that $a_j / \lambda_j \to \infty$. In particular, we can assume that $a_j > \lambda_j/ 2\sin(\pi/j)$ for all $j\in \mathbb{N}$. As the euclidean distance between two points on a circle with radius $r$ and angle $\varphi$ between them is given by $2\sin (\varphi/2)$, we therefore find $j$ balls with radius $\lambda_j$ and center points with distance $|a_j|$ to the origin that are pairwise disjoint. By the axisymmetry of $\bm{m}$ and the estimate \eqref{eq:proof_axisymmetric_blowup_1}, this implies together with the non-increasing energy along the flow that 
        \begin{equation*}
            \mathcal{E}(\bm{m}_0) \geq \mathcal{E}(\bm{m}(T_j)) \geq \frac{1}{2}\int_{\mathbb{R}^2} |\nabla\bm{m}(T_j)|^2 \, \mathrm{d}x \geq \frac{j \varepsilon}{2} \longrightarrow \infty \, , \, \text{~as~} \,  j\to \infty  \, ,
        \end{equation*} 
        a contradiction. Consequently, the sequence $a_j / \lambda_j$ is bounded and we can extract a subsequence such that $a_j / \lambda_j \to \xi$ for some $\xi \in \mathbb{R}^2$. Now, we define the new harmonic map $\tilde{\bm{\omega}}\colon \mathbb{S}^2 \to \mathbb{S}^2$ through its representation in stereographic coordinates by
        \begin{equation*}
            \tilde{\bm{\omega}}(x) = \bm{\omega}(x - \xi) \, .
        \end{equation*}
        We now claim 
        \begin{equation} \label{eq:proof_axisymmetric_blowup_3}
            \bm{m}(\lambda_j x) \longrightarrow \tilde{\bm{\omega}}(x) \, , \, \text{~as~} \,  j\to \infty \, ,
        \end{equation}
        in $L_\mathrm{loc}^\infty(\mathbb{R}^2, \mathbb{R}^3)$ and also $H_\mathrm{loc}^1(\mathbb{R}^2,\mathbb{R}^3)$. To prove this, we first define $y_j = \xi - a_j / \lambda_j$ and find that $y_j \to 0$ as $j\to \infty$. Let $U \subset \mathbb{R}^2$ be a compact set. Then there is an open set $V \supset U$ such that $\overline{V}$ is compact and for $j$ large enough we have $U+y_j \subset  V$. Then we gain for $x \in U$
        \begin{align*}
            &\phantom{{}={}}|\bm{m}(\lambda_j x) - \tilde{\bm{\omega}}(x)| = |\bm{m}(\lambda_j (x - \xi +  y_j) + a_j) - \tilde{\bm{\omega}}(x- \xi)| \\
            &\leq |\bm{m}(\lambda_j (x - \xi +  y_j) + a_j) - \tilde{\bm{\omega}}(x-\xi+y_j)| + |\tilde{\bm{\omega}}(x-\xi+y_j) - \tilde{\bm{\omega}}(x- \xi)| \\
            &\leq \|\bm{m}(\lambda_j \cdot + \,a_j) - \tilde{\bm{\omega}} \|_{L^\infty(V-\xi)} + \|\nabla\tilde{\bm{\omega}} \|_{L^\infty(V-\xi)} |y_j|  \longrightarrow 0 \, , \, \text{~as~} \,  j\to \infty \, , 
        \end{align*}  
        independently of $x$, by the convergence of $\bm{m}$ to $\tilde{\bm{\omega}}$ in $L_\mathrm{loc}^\infty(\mathbb{R}^2, \mathbb{R}^3)$ and the smoothness of $\tilde{\bm{\omega}}$. Thus, we have shown $L_\mathrm{loc}^\infty$-convergence. A very similar argument can be used to show $H_\mathrm{loc}^1$-convergence by using a substitution $x \mapsto x - y_j$ in the integrals.

        We show that $\tilde{\bm{\omega}}$ is axisymmetric. To this end, let $R \in O(3)_{\hat{\bm{e}}_3}$. For $x \neq -\hat{\bm{e}}_3$ we use the stereographic coordinates centered at the north pole. Then, there is $A \in O(2)$ representing the action $R.x = Ax$ in these coordinates. Using this, we conclude from the axisymmetry of $\bm{m}$ that
        \begin{align*}
            \bm{m}(\lambda_j x) &= \bm{m}_R(\lambda_j x) = R^{-1}\bm{m}(A \lambda_j x) = R^{-1}\bm{m}(\lambda_j A x)\\ 
            &\longrightarrow R^{-1}\tilde{\bm{\omega}}(Ax) = \tilde{\bm{\omega}}_R(x) \, , \, \text{~as~} \,  j\to \infty\, .
        \end{align*}
        By the uniqueness of the limit we conclude that $\tilde{\bm{\omega}}_R = \tilde{\bm{\omega}}$ on $\mathbb{S}^2 \setminus \{-\hat{\bm{e}}_3\}$ and by continuity on the whole sphere. Since $R$ was arbitrary, we have shown that $\tilde{\bm{\omega}}$ is axisymmetric and we express it in polar stereographic coordinates as
        \begin{equation*}
            \tilde{\bm{\omega}}(r, \varphi) = \bm{\Psi}(\beta(\theta(r)),\varphi)\, 
        \end{equation*}
        for some $\beta$ with parameter $a\neq 0$, where $\theta(r) = 2 \arctan(r)$. This proves the first part of the lemma.

        It remains to show the convergence of the profiles. To this end, let $0<\theta_0<\pi$ be arbitrary. The set expressed in spherical coordinates by $[0,\theta_0] \times [0,2\pi)$ corresponds to the compact set $\overline{B_{\tan(\theta_0/2)}}(0) \subset \mathbb{R}^2$ in stereographic coordinates. Hence, we find in view of \eqref{eq:proof_axisymmetric_blowup_3} the uniform convergence
        \begin{equation*}
            |\omega(\tan(\theta/2) \hat{\bm{e}}_1) - \bm{m}(\lambda_j \tan(\theta/2)\hat{\bm{e}}_1)| \longrightarrow 0 \, , \, \text{~as~} \,  j \to \infty \, 
        \end{equation*}
        on $[0,\theta_0]$. For $\theta \in [0,\theta_0]$ we define
        \begin{equation*}
            \theta_j = 2\arctan\left(\lambda_j\tan\left(\frac{\theta}{2}\right)\right) \, .
        \end{equation*}
        Then for the respective profiles $\beta$ and $h$ we conclude from the components
        \begin{align*}
            |\sin(h(\theta_j, T_j)) - \sin(\beta(\theta))| &\longrightarrow 0 \\
            |\cos(h(\theta_j, T_j)) - \cos(\beta(\theta))| &\longrightarrow 0 \, , \, \text{~as~} \,  j\to \infty \, ,
        \end{align*}
        independently of $\theta$. This implies that $h(\theta_j, T_j) \to \beta(\theta) + k\pi$ for some $k\in \mathbb{N}$, uniformly. By the continuity of the functions involved, $k$ is independent of $\theta$. 
    \end{proof}
\end{lemma}

We stress that the given proof works for any flow that features a bubbling phenomenon as in \eqref{eq:blowup_convergence}. For the HMHF there exist stronger results, e.g.,~\cite{jendrej2023,bertsch2011}. There, the authors prove a blowup tree result for the profiles themselves. We expect that one could reproduce these results for the SMHF as well. However, for our purposes, \cref{lemma:blow_up_pole} is sufficient.

\subsection{Comparison principle for profiles}

Parabolicity of \eqref{eq:SHE_h} provides the foundation for establishing an important tool: a so-called comparison principle. This principle ensures that solutions without initial crossing points remain non-crossing over time. Additionally, if one function is a super- or subsolution of \eqref{eq:SHE_h}, this property persists in a single direction. Similar results for the HMHF have been proved in~\cite{chang1991}.

\begin{lemma} \label{lemma:comparison_principle}
    Let $h_1 \in E_{m_1,n_1}$ be a solution of \eqref{eq:SHE_h}, and let $h_2 \in E_{m_2,n_2}$ satisfy 
    \begin{equation*}
        h_t \geq h_{\theta \theta}+\frac{\cos \theta}{\sin \theta} h_\theta-\frac{\sin 2 h}{2 \sin ^2 \theta}- \frac{\kappa}{2} \sin(2h-2\theta)\, .
    \end{equation*}
    Alternatively, let $h_1 \in E_{m_1,n_1}$ satisfy
    \begin{equation*}
        h_t \leq h_{\theta \theta}+\frac{\cos \theta}{\sin \theta} h_\theta-\frac{\sin 2 h}{2 \sin ^2 \theta}- \frac{\kappa}{2} \sin(2h-2\theta)\, ,
    \end{equation*}
    and let $h_2 \in E_{m_2,n_2}$ be a solution of \eqref{eq:SHE_h}. 
    
    In both cases, if $h_1(\theta,0) \leq h_2(\theta,0)$ for all $\theta \in [0,\pi]$, then $h_1(\theta,t) \leq h_2(\theta,t)$ for all $(\theta,t) \in [0,\pi]\times (0,T)$, where $T>0$ is the maximal existence time of the solutions.
    \begin{proof}
        We only prove the first case, as the second case can be shown analogously. First, $h_1(\theta,0) \leq h_2(\theta,0)$ implies $m_1 \leq m_2$ and $n_1 \leq n_2$. We define $\Delta = h_2 - h_1$ and obtain $\Delta_0 = h_2(0) - h_1(0) \geq 0$ and $\Delta(0,t) = (m_2 - m_1)\pi \eqqcolon m_0 \pi\geq 0$ and $\Delta(\pi,t) = (n_2 - n_1)\pi \eqqcolon n_0\pi\geq 0$ for all $t\in [0,T)$. By assumption on $h_1$ and $h_2$, we check that $\Delta$ satisfies 
        \begin{align*}
            0 &\leq \Delta_t - \Delta_{\theta\theta} - \frac{\cos \theta}{\sin \theta}\Delta_\theta + \frac{\sin 2h_2 - \sin 2h_1 }{2\sin^2\theta} + \frac{\kappa}{2} (\sin(2h_2-2\theta) - \sin(2h_1-2\theta))  \, \\
            &=  \Delta_t - \Delta_{\theta\theta} - \frac{\cos \theta}{\sin \theta}\Delta_\theta + \biggl(\underbrace{\frac{\sin\Delta}{\Delta} \frac{\cos(h_2+h_1)}{\sin^2\theta}}_{\eqcolon c_1(\theta,t)} + \underbrace{\kappa \frac{\sin\Delta}{\Delta} \cos(h_2+h_1-2\theta)\vphantom{\frac{\cos(h_2+h_1)}{\sin^2\theta}}}_{\eqqcolon c_2(\theta,t)}\biggr) \Delta \, .
        \end{align*}
        We want to apply \cref{lemma:max_principle_variant}. For this, let $\tilde{T} \in (0,T)$ be arbitrary. Then $\Delta$ certainly satisfies the inequality on $I_{\tilde{T}} = (0,\pi)\times(0,\tilde{T}]$ with initial data $\Delta_0 \geq 0$ and boundary values $\Delta(0,t) = m_0 \pi$ and $\Delta(\pi,t) = n_0 \pi$. We have to show that $c_1 + c_2$ is bounded from below.

        As a product of bounded functions, the function $c_2$ is bounded from below by $-\delta_2$ for some $\delta_2 > 0$. For $c_1$, we see that it is a product of bounded functions and $\sin^{-2}(\theta)$, which is positive on $(0, \pi)$ and diverges to $+\infty$ for $\theta \to 0$ and $\theta \to \pi$. Thus, it suffices to inspect the sign of
        \begin{equation*}
            f(\theta) \coloneqq \frac{\sin\Delta}{\Delta} \cos(h_2+h_1)
        \end{equation*}
        close to $0$ and $\pi$.        
        
        Now, four cases can occur: $m_0 = 0$ or $m_0 \geq 1$, and $n_0 = 0$ or $n_0 \geq 1$. In order to demonstrate the arguments, we consider the case $m_0 = 0$ and $n_0 \geq 1$, which can easily be transferred to the other cases due to the symmetry of \eqref{eq:SHE_h} with respect to the transform $\theta\mapsto \pi-\theta$. 
        
        To this end, we recall that $h_1$ and $h_2$ are continuous on $[0,\pi]\times[0,T)$ and thus uniformly continuous on $\overline{I}_{\tilde{T}}$. Therefore, there exists $0<\theta_1 < \frac{\pi}{8}$ such that for $i=1,2$ we have 
        \begin{equation*}
            |h_i(\theta, t) - m_i \pi| < \frac{\pi}{8} \quad \text{~for~all~} (\theta,t) \in [0,\theta_1]\times[0,\tilde{T}] \, .
        \end{equation*}
        This implies that
        \begin{equation*}
            |\Delta(\theta, t)| = |\Delta(\theta, t) - m_0 \pi| < \frac{\pi}{4} \quad \text{~for~all~} \theta \in [0,\theta_1]\times[0,\tilde{T}] \, .
        \end{equation*}
        Moreover, it also holds that
        \begin{equation*}
            |h_2(\theta, t) + h_1(\theta, t) - 2\theta - (m_1+m_2)\pi| \leq \frac{\pi}{2} \quad \text{~for~all~} \theta \in [0,\theta_1]\times[0,\tilde{T}] \, .
        \end{equation*}
        Consequently, using that we have $m_1 + m_2 = 2m_1 + m_0 = 2m_1$, we obtain
        \begin{equation*}
            f(\theta) = \frac{\sin \Delta}{\Delta} \cos(h_2+h_1 - 2\theta- (m_1+m_2)\pi) \geq 0 \quad \text{~for~all~} (\theta,t) \in [0,\theta_1]\times[0,\tilde{T}]  \, .
        \end{equation*}
        Furthermore, again by uniform continuity, there exists $0<\theta_2 < \pi$ such that for $i = 1,2$ 
        \begin{equation*}
            |h_i(\theta, t) - n_i \pi| < \frac{\pi}{4} \quad \text{~for~all~} (\theta,t) \in [\theta_2,\pi]\times[0,\tilde{T}] \, .
        \end{equation*}
        This yields
        \begin{equation*}
            \Delta(\theta, t) \geq n_0\pi - |\Delta(\theta, t) - n_0\pi| \geq \pi - |h_2(\theta,t) - n_2\pi| - |h_1(\theta,t) - n_1 \pi| > \frac{\pi}{2} \, ,
        \end{equation*}
        for all $(\theta,t) \in [\theta_2,\pi]\times[0,\tilde{T}]$. In particular, we have $\Delta(\theta_2, t) > 0$ for all $t\in [0,\tilde{T}]$. 

        The first estimate and the fact that $c_1$ is continuous on $[\theta_1,\theta_2]\times[0,\tilde{T}]$ then implies that there exists a global constant $\delta_1>0$ such that $c_1(\theta,t) > -\delta_1$ for all $(\theta,t)\in [0, \theta_2]\times (0,\tilde{T}]$. Thus, $c \coloneqq c_1 + c_2$ is bounded from below by $-(\delta_1 + \delta_2)$ on $[0,\theta_2]\times (0, \tilde{T})$. As a result, invoking \cref{lemma:max_principle_variant} we deduce $\Delta \geq 0$ on $[0,\theta_2]\times (0,\tilde{T})$ and since $\Delta > 0$ on $[\theta_2,\pi]\times (0,\tilde{T})$ we obtain $\Delta \geq 0$ on $I_{\tilde{T}}$. As $\tilde{T}$ was arbitrary, we conclude that $\Delta \geq 0$ on $I\times (0,T)$.
    \end{proof}
\end{lemma}

\section{Existence of Saddle Points} \label{chapter:saddlepoint}

\subsection{Hemispheric maps}

We notice that the energy is invariant under the additional transformation $\bm{m}_{\hat{A}} \mapsto \bm{m}\circ \hat{A}$, where $\hat{A}\colon \mathbb{S}^2\to\mathbb{S}^2$ is the antipodal map. This follows from the fact that $\nabla \bm{m}_{\hat{A}} = -(\nabla\bm{m})\circ \hat{A}$ and $\bm{\nu}\circ\hat{A} = -\bm{\nu}$.

In spherical coordinates the action has the form 
\begin{equation*}
    \hat{A}(\theta,\varphi) = \begin{cases}
        (\pi - \theta, \varphi+\pi) \quad \text{for~} \varphi \in [0,\pi) \, ,\\
        (\pi - \theta, \varphi-\pi) \quad \text{for~} \varphi \in [\pi,2\pi) \, .
    \end{cases} 
\end{equation*}
For axisymmetric $\bm{m}$, this yields for all $(\theta,\varphi) \in (0,\pi)\times[0,2\pi)$
\begin{equation*}
    \bm{m}_{\hat{A}}\left(\theta,\varphi \right) = \begin{pmatrix}
        \cos (\varphi+\pi) \sin (h(\pi - \theta))\\
     \sin (\varphi+\pi) \sin (h(\pi - \theta)) \\
     \cos (h(\pi - \theta) )
     \end{pmatrix}
    = \begin{pmatrix}
       \cos (\varphi) \sin(- h(\pi - \theta))\\
     \sin (\varphi) \sin(- h(\pi - \theta)) \\
     \cos (- h(\pi - \theta))
     \end{pmatrix} \, .  
\end{equation*}

\begin{figure}[t]
    \begin{center}
        \includegraphics{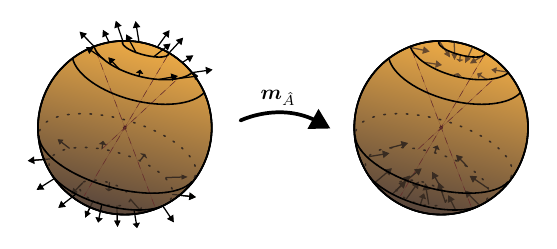}
        \caption{The action of $\bm{m}_{\hat{A}}$ is precisely the inversion of all base points of the vector field $\bm{m}$ at the origin. The corresponding vectors remain unchanged.} 
    \end{center}
\end{figure}

Thus, defining $h_{\hat{A}}(\theta)\coloneqq 2\pi k - h(\pi-\theta)$ for some $k\in \mathbb{Z}$, we see that the transformation $\bm{m} \mapsto \bm{m}_{\hat{A}}$ induces a transformation on the level of profiles of $h \mapsto h_{\hat{A}}$ and the axisymmetry of $\bm{m}$ is preserved under this transformation. 

\begin{definition} \label{def:hemispheric_map}
    We call a map $\bm{m}\colon \mathbb{S}^2 \to \mathbb{S}^2$ \emph{hemispheric} if and only if it is axisymmetric and satisfies $\bm{m} = \bm{m}_{\hat{A}}$. Similarly, we call a profile function $h\colon [0,\pi] \to \mathbb{R}$ a \emph{hemispheric profile} if and only if $h = h_{\hat{A}}$.

    We define the set of hemispheric profiles with boundary values $h(0) = m\pi$ and $h(\pi) = n\pi$ for $m,n\in \mathbb{Z}$ as
    \begin{equation*}
        H_{m,n} = \left\{ h \in E_{m,n} : h_{\hat{A}} = h \right\} \, . \qedhere
    \end{equation*} 
\end{definition}

For $h \in H_{m,n}$ the value of $k$ is determined by the end points
\begin{equation*} 
    k = \frac{m+n}{2} = \frac{h(\tfrac{\pi}{2})}{\pi} \, .
\end{equation*}
Consequently, hemispheric profiles are exactly the maps whose shift by $-k\pi$ is odd around the point $\frac{\pi}{2}$.

\begin{proposition} \label{prop:hemispheric_solution_preserved}
    Let $\bm{m}_0$ be hemispheric with hemispheric profile $h_0 \in H_{m,n}$ for some $m,n\in \mathbb{Z}$. Then the solution $\bm{m}$ to \eqref{eq:SHE} provided by \cref{theo:flow_existence} is also hemispheric for all $t\in(0,T)$, where $T>0$ is the maximal existence time until the first blowup, if present. In particular,  $h(\cdot,t) \in H_{m,n}$ for all $t\in(0,T)$.
\end{proposition}
The proof is analogous to the proof of \cref{prop:axisymmetric_solution_preserved}, using the fact that the equation is invariant under the transformation $\bm{m} \mapsto \bm{m}_{\hat{A}}$.

It follows immediately from this result that for hemispheric initial data it suffices to consider the initial value problem of \eqref{eq:SHE_h} on the half interval $[0,\frac{\pi}{2}]$ with boundary values $h(0) = m\pi$ and $h(\frac{\pi}{2}) = k\pi$ since $h$ is completely determined by its values on this interval.

We now consider initial hemispheric profiles $h_0 \in H_{1,1}$, then clearly $k=1$. Additionally, we obtain $Q(\bm{m}_0) =0$ for the corresponding field. 
\begin{proposition} \label{prop:no_blow_up}
    Let $\bm{m}_0$ be hemispheric with profile $h_0 \in H_{1,1}$ that satisfies the wedge condition
    \begin{equation} \label{eq:hemispheric_wedge_condition}
        \pi \leq h_0(\theta) \leq \pi + \theta\quad \text{~for~all~} \theta \in \left[0,\frac{\pi}{2}\right]  \, .
    \end{equation} 
    Then the solution $\bm{m}$ to \eqref{eq:SHE} provided by Theorem \ref{theo:flow_existence} is global and no blowup occurs at infinity. Therefore, $\bm{m}_\infty$ is a hemispheric solution to \eqref{eq:EL} with $Q(\bm{m}_\infty) = 0$. Moreover, $h_\infty$ satisfies \eqref{eq:hemispheric_wedge_condition}.

    \begin{proof}
        Assume, by contradiction, that $\bm{m}$ blows up at $T \in (0,\infty)$. Recall that, by \cref{prop:hemispheric_solution_preserved}, $\bm{m}$ is hemispheric for all $0< t < T$ and completely determined by its hemispheric profile $h$ on $(0,\frac{\pi}{2})$. We claim that 
        \begin{equation*}
            \pi \leq h(\theta, t) \leq \pi + \theta \text{~for~all~} (\theta,t)\in \left[ 0, \frac{\pi}{2} \right]\times (0, T) \, .
        \end{equation*}
        Since $(\theta,t) \mapsto \pi + \theta$ is a solution to \eqref{eq:SHE_h}, the upper bound follows immediately from \cref{lemma:comparison_principle}. For the lower bound, we adapt the statement of \cref{lemma:comparison_principle} to functions on the half interval $[0,\frac{\pi}{2}]$ as $h(\frac{\pi}{2},t) = \pi$ is fixed for all $t$. For $g(\theta,t)\equiv \pi$, we find
        \begin{equation*}
            0= g_t \leq g_{\theta \theta}+\frac{\cos \theta}{\sin \theta} g_\theta-\frac{\sin 2 g}{2 \sin ^2 \theta}- \frac{\kappa}{2} \sin(2g-2\theta) = \frac{\kappa}{2} \sin(2\theta)\, ,\, 
        \end{equation*}
        which is certainly true for all $(\theta,t) \in [0,\frac{\pi}{2}]\times (0,T)$. Therefore, the lower bound is also established.    
        
        By assumption on $T$ and by \cref{lemma:blow_up_pole}, there are $\lambda_k \searrow 0$ and $t_k \nearrow T$ such that for any $0<\theta<\pi$, if we define
        \begin{equation*}
            \theta_k = 2\arctan\left(\lambda_k\tan\left(\frac{\theta}{2}\right)\right) \, , 
        \end{equation*}
        then $h(\theta_k, t_k) \to \beta(\theta) + k\pi$ for some $k\in \mathbb{Z}$ and nonconstant profile $\beta$. By the already stated bounds we also obtain
        \begin{equation*}
            \pi \leq h(\theta_k, t_k) \leq \pi +\theta_k \longrightarrow \pi \, , \, \text{~as~} \,  k\to \infty \, ,
        \end{equation*}
        and hence $h(\theta_k, t_k) \to \pi \neq \beta(\theta) + k\pi$, as $\beta$ is nonconstant. Thus, by contradiction, $\bm{m}$ cannot blow up at $T$ and consequently $\bm{m}$ is global. Blowup at infinity can be ruled out by the same argument.  
    \end{proof}
\end{proposition}

    \begin{figure}[t]
        \begin{center}
            \includegraphics{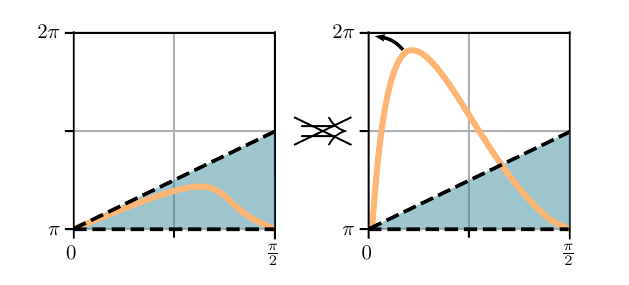}
            \caption{Pictogram of the proof idea of \cref{prop:no_blow_up}. Since the profile $h$ is confined in the wedge for all times, a blowup is prohibited as this requires the profile making a jump of at least $\pm \pi$ at $0$.}
        \end{center}
    \end{figure}

\begin{remark}
    An admissible profile that satisfies the condition is $h_0 \equiv \pi$. It is unclear whether all initial maps eventually flow into the same stationary solution. For $\kappa > 24$, a solution $\bm{m}_\infty$ obtained this way is however distinct from the minimizer of the energy $\mathcal{E}$ in the class of axisymmetric maps with degree 0. This follows from the fact that the profile $h$ of the minimizer satisfies $h < \pi$ on $(0,\pi)$~\cite[Prop. 3.1]{schroeder2024}.
\end{remark}

For convenience, we define the wedge $W_1\subset \mathbb{R}^2$ in which the graph of profiles satisfying condition \eqref{eq:hemispheric_wedge_condition} lies for all times. 
\begin{equation*}
    W_1 = \left\{ (x,y) \in \mathbb{R}^2 : 0\leq x\leq \frac{\pi}{2}, \pi \leq y \leq \pi + x \right\} \, .
\end{equation*}

From the wedge condition we can already deduce a uniform bound on the derivative of the profile $h$.

\begin{lemma} \label{lemma:bound_derivative_h}
    Let $\kappa \geq 4$ and let $h \in H_{1,1}$ be a solution of \eqref{eq:EL_h} that satisfies the wedge condition \eqref{eq:hemispheric_wedge_condition}. Then, $h'(\theta) \leq 1$ for all $\theta \in [0,\pi]$.
    \begin{proof}
        Since $h$ is a hemispheric profile, we deduce that $h'$ is an even function around $\frac{\pi}{2}$. Therefore, we focus on the interval $[0,\frac{\pi}{2}]$. By contradiction, assume that the maximal point $\theta_* \in [0,\frac{\pi}{2}]$ of $h'$ is such that $h'(\theta_*) >1$. From the wedge condition we conclude that $h'(0) \leq 1$ and $h'(\frac{\pi}{2}) \leq 0$. Thus, $\theta_* \in (0,\frac{\pi}{2})$ and $h''(\theta_*) = 0$. Then, as $h$ solves \eqref{eq:EL_h}, we obtain
        \begin{align*}
            0 &= 2\sin^2\theta_*\left(h'' + \frac{\cos\theta_*}{\sin\theta_*}h' - \frac{\sin 2 h}{2 \sin^2 \theta_*} - \frac{\kappa}{2} \sin(2 h-2\theta_*)\right) \\
            &> 2\cos\theta_* \sin\theta_* - \sin2h - \kappa \sin(2h-2\theta_*) \sin^2\theta_* \\
            &= \sin 2\theta_* - \sin2h - \kappa \sin(2h-2\theta_*) \sin^2\theta_* \eqqcolon f(\theta_*, h(\theta_*)) \, . \addtocounter{equation}{1}\tag{\theequation} \label{eq:definition_f}
        \end{align*}
        Using trigonometric identities, it can be shown that the function $f\colon \mathbb{R}^2 \to \mathbb{R}$, $(x,y) \mapsto f(x,y)$ is nonnegative on the set $W$ for $\kappa \geq 4$. Hence, as by the wedge condition $(\theta_*, h(\theta_*)) \in W_1$, we conclude that $f(\theta_*, h(\theta_*)) \geq 0$, contradicting \eqref{eq:definition_f}.
    \end{proof}
\end{lemma}

\begin{cor} \label{cor:monotonicity_h}
    Let $\kappa \geq 4$ and let $h \in H_{1,1}$ be a solution of \eqref{eq:EL_h} that satisfies the wedge condition \eqref{eq:hemispheric_wedge_condition}. Then, the function $\theta \mapsto \pi+\theta-h(\theta)$ is monotonically increasing. The function $\theta \mapsto \sin^2(h-\theta) \sin\theta$ is monotonically increasing on the interval $[0,\frac{\pi}{2}]$ and monotonically decreasing on the interval $[\frac{\pi}{2},\pi]$.
    \begin{proof}
        The first statement follows from the previous lemma. The second one follows from the lemma and from the fact that $\sin(2h-2\theta) \leq 0$ for $\theta \in [0,\frac{\pi}{2}]$ and $\sin(2h-2\theta) \geq 0$ for $\theta \in [\frac{\pi}{2},\pi]$ due to the wedge condition. 
    \end{proof}
\end{cor}

In a similar fashion, we prove the non-blowup of another class of hemispheric solutions with initial profiles $h_0 \in H_{0,2}$. For the corresponding initial field, we assert that $Q(\bm{m}_0) = 0$. This follows from \eqref{eq:Q_axisymmetric} and the fact that $h_0(0)= 0$ and $h_0(\pi) = 2\pi$.
\begin{proposition}\label{rema:hemispheric_wedge_condition_3}
    Let $\kappa \geq 4$ and let $\bm{m}_0$ be hemispheric with profile $h_0 \in H_{0,2}$ that satisfies the wedge condition
    \begin{equation} \label{eq:hemispheric_wedge_condition_2}
        \theta \leq h_0(\theta) \leq 2\theta\quad \text{~for~all~} \theta \in \left[0,\frac{\pi}{2}\right]  \, .
    \end{equation} 
    Then the solution $\bm{m}$ to \eqref{eq:SHE} provided by Theorem \ref{theo:flow_existence} is global and no blowup occurs at infinity. Therefore, $\bm{m}_\infty$ is a hemispheric solution to \eqref{eq:EL} with $Q(\bm{m}_\infty) = 0$.  Moreover, $h_\infty$ satisfies \eqref{eq:hemispheric_wedge_condition_2}.
    \begin{proof}
        We check if the lower and upper bound in \eqref{eq:hemispheric_wedge_condition_2} are sub- and supersolutions of \eqref{eq:SHE_h} on the half interval, respectively. The lower bound $(\theta,t) \mapsto \theta$ is a solution of \eqref{eq:SHE_h}. For the upper bound we set $g(\theta,t) = 2\theta$. As computed in \cref{prop:kappa_4_solution}, $g$ satisfies by the assumption $\kappa \geq 4$ the inequality
        \begin{equation*}
             g_{\theta \theta}+\frac{\cos \theta}{\sin \theta} g_\theta-\frac{\sin 2 g}{2 \sin ^2 \theta}- \frac{\kappa}{2} \sin(2g-2\theta) = \left( 2 - \frac{\kappa}{2} \right) \sin(2\theta)  \leq 0 = g_t\, ,
        \end{equation*}
        for all $(\theta,t) \in [0,\frac{\pi}{2}]\times (0,T)$. Hence, $g$ is a supersolution of \eqref{eq:SHE_h} and we can argue as in the previous proof to rule out blowup of $\bm{m}$ in finite or infinite time.
    \end{proof}
\end{proposition}

We define the set in which the graph of profiles satisfying condition \eqref{eq:hemispheric_wedge_condition_2} lies for all times as
\begin{equation*}
    W_2 \coloneqq \left\{ (x,y) \in \mathbb{R}^2 : 0\leq x\leq \frac{\pi}{2}, x \leq y \leq 2 x \right\} \, .
\end{equation*}

As in the first case, we obtain a uniform bound on the derivative, complementary to that of \cref{lemma:bound_derivative_h}.
\begin{lemma} \label{lemma:bound_derivative_h_3}
    Let $\kappa \geq 4$ and let $h \in H_{0,2}$ be a solution of \eqref{eq:EL_h} that satisfies the wedge condition \eqref{eq:hemispheric_wedge_condition_2}. Then, $h'(\theta) \geq 1$ for all $\theta \in [0,\pi]$.
    \begin{proof}
        We almost repeat the proof of \cref{lemma:bound_derivative_h} verbatim, only changing the sets in which the graph lies and interchanging the signs in the inequalities. For the argument by contradiction, we now assume that $\theta_* \in (0,\frac{\pi}{2})$ is the minimal point of $h'$ with $h'(\theta_*) <1$ and $h''(\theta_*) = 0$. Then, as $h$ is a solution to \eqref{eq:EL_h}, we obtain $0 < f(\theta_*,h(\theta_*))$, where $f$ is defined as in \eqref{eq:definition_f}. In this case, $(\theta_*, h(\theta_*)) \in W_2$ and $\kappa \geq 4$ imply $f(\theta_*, h(\theta_*)) \leq 0$, a contradiction.
    \end{proof}
\end{lemma}

In conclusion, there exist two types of solutions of \eqref{eq:EL} that obey the antipodal symmetry. In the case of the harmonic map equation between two-spheres, the only hemispheric solution is the constant map.

\subsection{Saddle points of first type. Proof of 
Theorem 1} \label{section:proof_saddle_point_1}

First, we define the profile of the initial map for which we then gain an energy bound. For $\kappa > 0$ we let $0<\theta_0(\kappa)<\frac{\pi}{2}$, chosen later, and define the initial profile as
\begin{equation} \label{eq:definition_h_0kappa}
    h_{0,\kappa}(\theta) = \begin{cases}
        \pi + \theta & \text{for } \theta \in [0,\theta_0(\kappa)] \, ,\\
        \pi - \frac{\theta_0(\kappa)}{ \frac{\pi}{2} - \theta_0(\kappa)}\left( \theta - \frac{\pi}{2}\right) & \text{for } \theta \in (\theta_0(\kappa), \pi - \theta_0(\kappa)] \, ,\\
        \theta & \text{for } \theta \in (\pi - \theta_0(\kappa), \pi] \, .
    \end{cases}
\end{equation}
This profile belongs to $H_{1,1}$ and satisfies the wedge condition \eqref{eq:hemispheric_wedge_condition}. Thus, $h_{0,\kappa}$ qualifies as a valid initial profile for the flow in \cref{prop:no_blow_up}.

\begin{figure}[t]
    \begin{center}
        \includegraphics{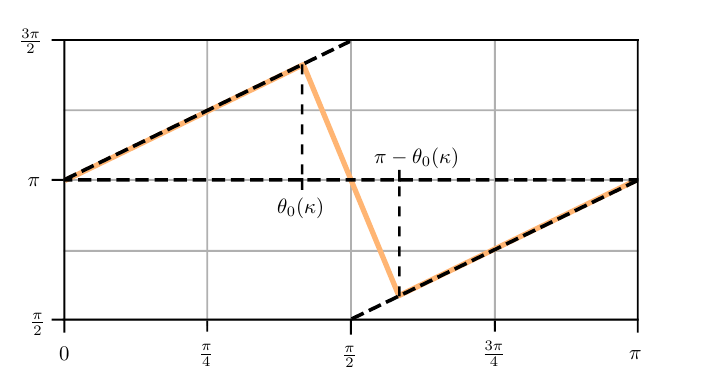}
        \caption{The function $h_{0,\kappa}$ in dependence of $\theta_0(\kappa)$.}
    \end{center}
\end{figure}

\begin{lemma} \label{lemma:energy_bound_h0}
    For $\kappa \geq 4$, there exists $\theta_0(\kappa)$ such that $h_{0,\kappa}$ satisfies the energy bound
    \begin{equation*}
        E(h_{0,\kappa}) \leq C\sqrt{\kappa} 
    \end{equation*}
    for some constant $C > 0$ independent of $\kappa$.

    \begin{proof}
        Since $h_{0,\kappa}$ is a hemispheric profile, it suffices to compute twice the energy on the interval $[0,\frac{\pi}{2}]$. Setting $h = h_{0,\kappa}$ and $\delta = \frac{\pi}{2\sqrt{\kappa}}$, we define $\theta_0(\kappa) = \theta_0 \coloneqq \frac{\pi}{2} -\delta$. As $\kappa \geq 4$, we assert that $\delta \leq \tfrac{\pi}{4}$. Moreover, we set
        \begin{equation*}
            a = \frac{\theta_0}{\frac{\pi}{2}-\theta_0} = \frac{\frac{\pi}{2}-\delta}{\delta} \, .
        \end{equation*}
        Doing so, we gain
        \begin{align*}
            E(h) &=\phantom{+} \underbrace{\int_0^{\theta_0} \sin\theta + \frac{\sin^2(\pi+\theta)}{\sin\theta} + \kappa \sin^2(\pi + \theta - \theta)\sin\theta \, \mathrm{d}\theta}_{\eqqcolon I_1}\\
            &\phantom{=} + \underbrace{\int_{\mathrlap{\theta_0}}^{\frac{\pi}{2}} a^2 \sin\theta + \frac{\sin^2\left(\pi - a(\theta - \frac{\pi}{2})\right)}{\sin\theta} + \kappa \sin^2\left(\pi - a(\theta - \tfrac{\pi}{2})-\theta\right)\sin\theta \, \mathrm{d} \theta}_{\eqqcolon I_2} \, .
        \end{align*}
        By definition of $\theta_0$, we also have $\cos\theta_0 = \sin\delta$ and $\sin\theta_0 = \cos\delta$. Thus, 
        \begin{equation*}
            I_1 = 2(1-\sin\delta) \, .
        \end{equation*}
        For $I_2$, we first estimate $\sin^{-1}\theta \leq \sin^{-1}\theta_0 \leq \sqrt{2}$ for $\theta \in [\theta_0,\frac{\pi}{2}] \subset [\tfrac{\pi}{4},\frac{\pi}{2}]$. Consequently, we obtain the bound
        \begin{equation*}
            I_2 \leq a^2\delta + \sqrt{2}\delta + \kappa\delta = \frac{\pi^2}{4\delta} + \kappa\delta - \pi + \delta \left( 1 + \sqrt{2} \right) \, ,
        \end{equation*}
        after inserting $a$. We combine the estimates to retrieve
        \begin{equation*}
            E(h) \leq \frac{\pi^2}{4\delta} + \delta \kappa + \underbrace{2 - \pi + \delta \left( 1 + 1\sqrt{2} - 2 \frac{\sin\delta}{\delta}\right)}_{\eqqcolon f(\delta)} \, .
        \end{equation*}
        As $\delta \leq \tfrac{\pi}{4}$, we have $f(\delta)<0$, which can easily be verified. The desired energy bound then follows directly upon inserting the definition of $\delta$ into the other two remaining terms.
    \end{proof}
\end{lemma}

In the following, we let $h \coloneqq h_\kappa$ be the smooth profile obtained from the initial profile $h_{0,\kappa}$, as in \cref{prop:no_blow_up}, for a given $\kappa \geq 4$. This implies that $h$ satisfies the wedge condition \eqref{eq:hemispheric_wedge_condition}, solves \eqref{eq:EL_h}, and that, as the energy decreases along the flow according to \cref{theo:flow_existence}, we immediately obtain with \cref{lemma:energy_bound_h0} the energy bound
\begin{equation} \label{eq:energy_bound}
    E(h) \leq C\sqrt{\kappa} \, .
\end{equation}
From this bound we get the following pointwise convergence.
\begin{lemma} \label{lemma:pointwise_convergence_h}
    For any $0<\theta_*<\frac{\pi}{2}$, we obtain that $h_\kappa(\theta) \to \pi + \theta$ as $\kappa \to \infty$, uniformly on $[0,\theta_*]$. Likewise, we have that $h_\kappa(\theta) \to \theta$ as $\kappa \to \infty$, uniformly on $[\pi - \theta_*, \pi]$.
    \begin{proof}
        We only show the first statement as the other one follows from  hemispheric symmetry. Assume, by contradiction, that there exists a $0<\theta_*<\frac{\pi}{2}$ such that for a sequence $\kappa \to \infty$ there exist $\theta_\kappa \in [0,\theta_*]$ such that $\pi + \theta_\kappa - h_\kappa(\theta_\kappa) \geq \varepsilon$ for some $0 < \varepsilon < \frac{\pi}{2}$. By \cref{cor:monotonicity_h}, this implies for all $\theta \in [\theta_*, \frac{\pi}{2}]$ that
        \begin{equation*}
            \sin^2(h_\kappa(\theta) - \theta)\sin\theta \geq \sin^2(h_\kappa(\theta_*) - \theta_*)\sin\theta_* \geq \sin^2\varepsilon\, \sin\theta_* > 0 \, .
        \end{equation*} 
        Invoking now the energy bound \eqref{eq:energy_bound} we find that
        \begin{equation*}
            \pi \sqrt{2\kappa} \geq E(h_\kappa) \geq \frac{\kappa}{2}\int_{\theta_*}^{\frac{\pi}{2}} \sin^2(h_\kappa(\theta) - \theta)\sin\theta \, \mathrm{d}\theta \geq \frac{\kappa}{2} (\tfrac{\pi}{2} -\theta_*) \sin^2\varepsilon \, \sin\theta_*  \, .
        \end{equation*}
        We arrive at a contradiction for $\kappa \to \infty$.
    \end{proof}   
\end{lemma}

We observe that as $\kappa \to \infty$, the profile $h_\kappa$ approaches a sawtooth shape, exhibiting an increasingly sharp transition near $\theta = \frac{\pi}{2}$. The next lemma quantifies the growth rate of the magnitude of the derivative at this point. Note that, by the wedge condition, we have $h'(\frac{\pi}{2}) \leq 0$.

\begin{lemma} \label{lemma:bound_derivative_h_2}
    There exists a universal constant $C>0$ such that
    \begin{equation*}
        -h'(\tfrac{\pi}{2}) \leq C \sqrt{\kappa} 
    \end{equation*}
    for all $\kappa \geq 4$.
    \begin{proof}
        Let $0< \delta \coloneqq \delta(\kappa) = 1/\sqrt{\kappa}$ and $\theta_\delta = \frac{\pi}{2} - \delta$. Then, adding a zero and using the fundamental theorem of calculus, we compute
        \begin{align*}
            - h'(\tfrac{\pi}{2}) \delta &= - \int_{\theta_\delta}^\frac{\pi}{2} h'(\tfrac{\pi}{2}) \sin(\tfrac{\pi}{2}) - h'(\theta) \sin\theta \, \mathrm{d}\theta - \int_{\theta_\delta}^\frac{\pi}{2} h'(\theta) \sin\theta \, \mathrm{d} \theta \\
            &= -\int_{\theta_\delta}^\frac{\pi}{2} \int_\theta^\frac{\pi}{2} \left(h' \sin\right)'(\tau) \, \mathrm{d}\tau \mathrm{d}\theta - \int_{\theta_\delta}^\frac{\pi}{2} h'(\theta) \sin\theta \, \mathrm{d}\theta \\
            &= \underbrace{-\int_{\theta_\delta}^\frac{\pi}{2} \int_\theta^\frac{\pi}{2} \frac{\sin 2h}{2\sin \tau} + \frac{\kappa}{2} \sin(2h-2\tau)\sin\tau  \, \mathrm{d} \tau \mathrm{d}\theta}_{\eqqcolon I_1}  \underbrace{-\int_{\theta_\delta}^\frac{\pi}{2}  h'(\theta) \sin\theta \, \mathrm{d}\theta }_{\eqcolon I_2} \, , 
        \end{align*} 
        where we used the fact that $h$ is a solution to \eqref{eq:EL_h}. Now, for $I_1$ we recall that $\sin 2h \geq 0$ and $\sin(2h-2\theta) \leq 0$ as $(\theta, h(\theta)) \in W_1$. Therefore, we estimate the first term by means of Jensen's inequality as
        \begin{align*}
            I_1 &\leq - \frac{\kappa}{2}\int_{\theta_\delta}^\frac{\pi}{2} \int_\theta^\frac{\pi}{2} \sin(2h-2\tau)\sin\tau \, \mathrm{d} \tau \mathrm{d}\theta \\
            &=  \kappa \int_{\theta_\delta}^\frac{\pi}{2} \int_\theta^\frac{\pi}{2} \smash{\underbrace{(-\cos(h-\tau))}_{\leq 1}}\sin(h-\tau) \sin\tau \, \mathrm{d} \tau \mathrm{d}\theta \\
            &\leq \sqrt{2\kappa \delta} \int_{\theta_\delta}^\frac{\pi}{2} \left( \frac{\kappa}{2} \int_\theta^\frac{\pi}{2} \sin^2(h-\tau) \sin\tau \, \mathrm{d}\tau \right)^{1/2} \mathrm{d}\theta \\
            &\leq \sqrt{2E(h)\delta} \sqrt{\kappa} \delta = \sqrt{2E(h) \delta} \, ,
        \end{align*}
        taking into account $\delta = 1/\sqrt{\kappa}$. For $I_2$ we obtain also by Jensen's inequality 
        \begin{equation*}
            I_2 \leq \sqrt{2\delta}\left( \frac{1}{2}\int_{\theta_\delta}^\frac{\pi}{2} h'(\theta)^2 \sin\theta \, \mathrm{d}\theta \right)^{1/2} \leq \sqrt{2E(h)\delta} \, .
        \end{equation*}
        We combine the estimates to retrieve
        \begin{equation*}
            - h'(\tfrac{\pi}{2}) \leq 2\sqrt{2\frac{E(h)}{\delta}} = 2\sqrt{2E(h) \sqrt{\kappa}} \, .
        \end{equation*}
        Finally, using \eqref{eq:energy_bound}, we arrive at the statement.
    \end{proof}
\end{lemma}

In order to prove \cref{theo:saddle_point}, we construct a perturbation $h_\varepsilon = h + \varepsilon g$ of the profile $h$ for some $g \in E_{0,0}$ such that the resulting map has lower energy. To this end, in light of \eqref{eq:energy_difference_hg_eps}, we need to show that the second variation at $h$ in direction $g$ is negative. We define $g = g_\kappa\colon [0,\pi] \to \mathbb{R}$ as the smooth map
\begin{equation} \label{eq:definition_g}
    g_\kappa(\theta) = (h_\kappa'(\theta)-1)\sin\theta \, ,
\end{equation}
corresponding to an infinitesimal shift at the equator in altitudinal direction. Then, we have $g(0) = g(\pi) = 0$, which implies that $g$ is a valid perturbation. We compute the first derivative of $g$ as
\begin{equation} \label{eq:derivative_g}
    g' = h'' \sin\theta + h' \cos\theta - \cos\theta = \left(\frac{\sin 2 h}{2\sin^2 \theta}+ \frac{\kappa}{2} \sin(2 h-2 \theta)\right)\sin \theta -\cos\theta \, ,
\end{equation}
where we again used that $h$ is a solution to \eqref{eq:EL_h}. As a consequence of \cref{lemma:bound_derivative_h}, we find that $g(\theta) \leq 0$ for all $\theta \in [0,\pi]$. Moreover, we observe that $2 g'(\theta)\sin \theta$ is exactly equal to $-f(\theta, h(\theta))$ in the proof of said lemma. As $f \geq 0$ on $W_1$, we conclude by the wedge condition that $g' \leq 0$ on $[0, \frac{\pi}{2}]$ and by symmetry $g' \geq 0$ on $[\frac{\pi}{2}, \pi]$. Consequently, $g$ is nonpositive on $[0,\pi]$ and attains its minimum at $\frac{\pi}{2}$. 

\begin{proposition} \label{prop:second_variation_negative}
    There exists $\kappa_0 \geq 4$ such that for all $\kappa \geq \kappa_0$ we have
    \begin{equation*}
        \delta^2 E_\kappa[h_\kappa](g_\kappa) < 0 \, .
    \end{equation*}
    \begin{proof}
        Recalling \eqref{eq:second_variation_E}, we rewrite the two terms not involving $g'$ by inserting the definition of $g$ once and then integrating by parts. For the first one we obtain
        \begin{align*}
            \int_0^\pi \frac{\cos 2h}{\sin^2\theta} g^2 \sin\theta \, \mathrm{d} \theta &= \int_0^\pi \cos 2h(h' -1)g \, \mathrm{d} \theta \\
            &= \int_0^\pi - \frac{\sin 2h}{2\sin\theta} g' \sin\theta - \cos2h \; g  \, \mathrm{d} \theta \, .
        \end{align*}
        Likewise, we compute the second term as 
        \begin{align*}
            \int_0^\pi \kappa \cos(2h-2\theta) g^2 \sin\theta \, \mathrm{d} \theta &= \int_0^\pi \kappa \cos(2h-2\theta)(h' -1)g \sin^2\theta \, \mathrm{d} \theta \\
            &=\int_0^\pi - \frac{\kappa}{2} \sin(2h-2\theta) (g' \sin^2\theta + 2g \sin\theta \cos\theta) \, \mathrm{d} \theta \, .
        \end{align*}
        Consequently, we rewrite the second variation using \eqref{eq:derivative_g} as
        \begin{align*}
            \delta^2 E[h](g) &= \int_0^\pi {g'}^2 \sin \theta - \biggl(\underbrace{ \vphantom{\frac{1}{\sin^2}}\frac{\sin 2h}{2\sin\theta} + \frac{\kappa}{2} \sin(2h-2\theta) \sin\theta }_{=\, g' + \cos\theta}\biggr) g' \sin\theta  \, \mathrm{d} \theta\\
            &\phantom{{}= } + \int_0^\pi -\cos 2h \, g - \kappa \sin(2h-2\theta) g \sin\theta \cos\theta \, \mathrm{d} \theta \\ 
            &= \int_0^\pi - \frac{1}{2}\sin 2\theta\,g'-\cos 2h \, g - \kappa \sin(2h-2\theta) g \sin\theta \cos\theta \, \mathrm{d} \theta \\ 
            &= 2 \int_0^{\frac{\pi}{2}} \left(\cos2\theta-\cos 2h-  \kappa \sin(2h-2\theta) \sin\theta \cos\theta \right) g  \, \mathrm{d} \theta \, , \addtocounter{equation}{1}\tag{\theequation} \label{eq:second_variation_g}
        \end{align*}
        where we have used that the integrand is even at $\frac{\pi}{2}$. We decrease the domain of integration even further. To this end, we define the function $\lambda\colon \mathbb{R}^2 \to \mathbb{R}$ by
        \begin{equation} \label{eq:lambda_function}
            \lambda(x,y) = \cos 2x - \cos 2y - \kappa \sin(2y-2x) \sin x \cos x \, ,
        \end{equation}
        and see that the term in parentheses in the integrand is exactly equal to $\lambda(\theta, h(\theta))$. Using trigonometric identities, one shows that $\lambda$ is nonnegative on the set $W_1 \cap \{ 0\leq x\leq \tfrac{\pi}{4}\}$, as long as $\kappa \geq 4$. Hence, as $g \leq 0$, we deduce that the integral from $0$ to $\frac{\pi}{4}$ in \eqref{eq:second_variation_g} is negative. Therefore, we obtain
        \begin{equation}\label{eq:bound_delta2E}
            \delta^2 E[h](g) < 2 \biggl( \underbrace{\int_{\frac{\pi}{4}}^{\frac{\pi}{2}} \cos2\theta \, g \, \mathrm{d} \theta}_{\eqcolon I_1} + \underbrace{\int_{\frac{\pi}{4}}^{\frac{\pi}{2}} - \cos 2h \, g \, \mathrm{d} \theta}_{\eqcolon I_2} -  \kappa \underbrace{\int_{\frac{\pi}{4}}^{\frac{\pi}{2}} \sin(2h-2\theta) g \sin\theta \cos\theta \, \mathrm{d} \theta}_{\eqcolon I_3} \biggr) \, . 
        \end{equation} 

        We begin with the bound for $I_1$. We insert the definition of $g$ and integrate by parts, which yields
        \begin{align*}
            I_1 &= \int_{\frac{\pi}{4}}^{\frac{\pi}{2}} \cos2\theta (h' -1) \sin\theta \, \mathrm{d} \theta \\
            &= \frac{\pi}{2} + \int_{\frac{\pi}{4}}^{\frac{\pi}{2}} \underbrace{(\pi+\theta-h)}_{\geq 0} \underbrace{\left( \cos\theta \cos 2\theta - 2 \sin\theta \sin 2\theta \right)}_{\leq 0} \, \mathrm{d} \theta \leq \frac{\pi}{2} \, . \addtocounter{equation}{1}\tag{\theequation} \label{eq:bound_I1}
        \end{align*}
        For $I_2$ we insert again the definition of $g$ and integrate by parts. Using the fact that
        \begin{equation*}
            \int_{\frac{\pi}{4}}^{\frac{\pi}{2}} \frac{1}{2} \sin 2\theta \cos \theta + \cos2\theta \sin \theta \, \mathrm{d} \theta = - \frac{1}{2\sqrt{2}}\, ,
        \end{equation*}
        we obtain 
        \begin{align*}
        I_2 &= -\int_{\frac{\pi}{4}}^{\frac{\pi}{2}} \cos2h (h' -1) \sin\theta \, \mathrm{d} \theta \\
            &= - \frac{\sin 2h}{2} \sin \theta \bigg|_{\frac{\pi}{4}}^{\frac{\pi}{2}} + \int_{\frac{\pi}{4}}^{\frac{\pi}{2}} \frac{\sin 2h}{2} \cos\theta + \cos 2h \sin \theta \, \mathrm{d} \theta \\
        &= \frac{\sin (2h(\tfrac{\pi}{4})) -1}{2\sqrt{2}} + \int_{\frac{\pi}{4}}^{\frac{\pi}{2}} \frac{1}{2} \left( \sin 2h - \sin 2\theta \right) \cos \theta + \left( \cos 2h - \cos 2\theta \right) \sin \theta \, \mathrm{d} \theta \, .
        \end{align*}
        The first term is nonpositive. Using a trigonometric identity for the integrand in the second term, we thus get
        \begin{align*}
            I_2 &\leq \frac{1}{2} \int_{\frac{\pi}{4}}^{\frac{\pi}{2}} \sin (h-\theta)(3 \cos (h+2 \theta)-\cos h) \, \mathrm{d} \theta \\
            &\leq \frac{C}{\sqrt{\kappa}} \Biggl( \underbrace{ \frac{\kappa}{2} \int_{ \frac{\pi}{4}}^{\frac{\pi}{2}} \sin^2(h-\theta)\sin\theta \, \mathrm{d} \theta}_{\leq E(h)} \Biggr)^{1/2} \leq C\sqrt{ \frac{ E(h)}{\kappa}} \leq \frac{C}{\sqrt[4]{\kappa}}  \, , \addtocounter{equation}{1}\tag{\theequation} \label{eq:bound_I2}
        \end{align*}
        where we used Jensen's inequality and \eqref{eq:energy_bound}.

        For the bound on $I_3$, we define by virtue of the intermediate value theorem the point $\theta_\kappa \in (\frac{\pi}{4}, \tfrac{\pi}{2})$ to be the first point such that
        \begin{equation*}
            h(\theta_\kappa) = \frac{3\pi}{4} + \theta_\kappa \, ,
        \end{equation*}
        see \cref{fig:wedge6}. From \cref{lemma:bound_derivative_h} we infer that once $h$ has crossed the line $\frac{3\pi}{4} + \theta$, it can never surpass it again since that would require $h' > 1$. 
        \begin{figure}[t]
            \begin{center}
                \includegraphics{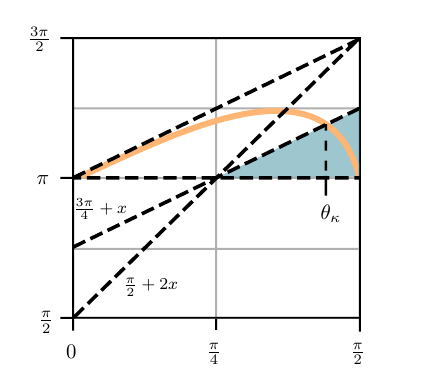}
                \caption{Plot showing the position of $\theta_\kappa$. For $\theta \geq \theta_\kappa$ the graph clearly lies below the line $y = \tfrac{\pi}{2} + 2x$.}
                \label{fig:wedge6}
            \end{center}
        \end{figure}
        
        In view of \cref{lemma:pointwise_convergence_h}, we observe that there exists $\kappa_0\geq 4$ such that for all $\kappa \geq \kappa_0$ we can guarantee $\theta_\kappa \in [\frac{3\pi}{8},\frac{\pi}{2}]$. We note that, by \cref{lemma:bound_derivative_h} and the wedge condition, the integrand in $I_3$ is nonnegative. Hence, decreasing the domain of integration to $[\frac{\pi}{4}, \theta_\kappa]$ decreases the value of the integral. Additionally, $\sin^2\theta \cos\theta \geq \sin^2\theta_\kappa \cos\theta_\kappa$ for all $\theta\in [\frac{\pi}{4}, \theta_\kappa]$, as $\theta_\kappa \geq \frac{3\pi}{8}$. Putting this together, we get the bound
        \begin{align*}
            I_3 &= \int_{\frac{\pi}{4}}^{\frac{\pi}{2}} \sin(2h-2\theta) (h' -1) \sin^2\theta \cos\theta \, \mathrm{d} \theta \\
            &\geq \sin^2\theta_\kappa \cos\theta_\kappa \int_{\frac{\pi}{4}}^{\theta_\kappa} 2\sin(h-\theta)\cos(h-\theta) (h' -1) \, \mathrm{d} \theta \\
            &= \sin^2\theta_\kappa \cos\theta_\kappa \left(  \sin^2\bigl(h(\theta_\kappa)-\theta_\kappa\bigr) - \sin^2\bigl(h(\tfrac{\pi}{4})-\tfrac{\pi}{4}\bigr)\right) \, .
        \end{align*}
        By definition of $\theta_\kappa$ we have $\sin^2(h(\theta_\kappa)-\theta_\kappa) = \frac{1}{2}$ for all $\kappa$. Moreover, in view of \cref{lemma:pointwise_convergence_h}, we have $\sin^2(h(\tfrac{\pi}{4})-\tfrac{\pi}{4}) \to 0$ as $\kappa \to \infty$. Defining $\delta_{\kappa} = \tfrac{\pi}{2} - \theta_\kappa$, such that $\delta_\kappa \to 0$ as $\kappa \to \infty$, we obtain
        \begin{equation} \label{eq:bound_I3}
            I_3 \geq \sin\delta_{\kappa} \cos^2\delta_{\kappa} \left( \frac{1}{2}  - \sin^2\bigl(h(\tfrac{\pi}{4})-\tfrac{\pi}{4}\bigr) \right) = \delta_\kappa \left( \frac{1}{2} + o(1) \right) \, ,
        \end{equation}
        where $o(1)$ is with respect to $\kappa \to \infty$. 
        
        We claim that there exists a constant $C>0$ such that
        \begin{equation} \label{eq:delta_bound}
            \frac{1}{C\sqrt{\kappa}}  \leq \delta_{\kappa} \leq \frac{C}{\sqrt{\kappa}} \, 
        \end{equation}
        for all $\kappa \geq \kappa_0$, with $\kappa_0$ increased if needed. For the upper bound we argue as in the proof of \cref{lemma:pointwise_convergence_h}. Using \cref{cor:monotonicity_h} and the energy bound \eqref{eq:energy_bound} we find that
        \begin{equation*}
            C \sqrt{\kappa} \geq E(h) \geq \frac{\kappa}{2} \int_{\theta_\kappa}^{\frac{\pi}{2}} \sin^2(h-\theta)\sin\theta \, \mathrm{d} \theta \geq \frac{\kappa}{2} \sin^2(h(\theta_\kappa)-\theta_\kappa)\sin\theta_\kappa \delta_{\kappa} \geq \frac{\kappa\delta_{\kappa}}{4\sqrt{2}} ,
        \end{equation*}  
        where we used $\sin\theta_\kappa > 1/\sqrt{2}$ as $\theta_\kappa > \frac{\pi}{4}$. Hence, for all $\kappa\geq \kappa_0$
        \begin{equation*}
            \delta_{\kappa} \leq \frac{C}{\sqrt{\kappa}} \, .
        \end{equation*}
        For the lower bound we use the definition of $\theta_\kappa$. By Taylor's theorem we have
        \begin{equation*}
            \frac{\pi}{4} - \delta_{\kappa} = h(\theta_\kappa) - h(\tfrac{\pi}{2}) = -h'(\tfrac{\pi}{2}) \delta_{\kappa} + \int_{\theta_\kappa}^{\frac{\pi}{2}} h''(\theta) (\theta-\theta_\kappa) \, \mathrm{d}\theta \, .
        \end{equation*}
        Now, as for all $\theta \in [\theta_\kappa,\frac{\pi}{2}]$ we have $(\theta,h(\theta)) \in W_1 \cap \{ y \leq \tfrac{\pi}{2} + 2x \}$ (see \cref{fig:wedge6}), this implies, using a trigonometric identity, that for these $\theta$ we have
        \begin{align*}
            &\phantom{{}={}} \frac{\sin 2 h}{2\sin^2 \theta}+ \frac{\kappa}{2} \sin(2 h-2 \theta) \\
            &= \frac{1}{8 \sin^2\theta} \Bigl( 2\kappa \underbrace{\sin(2h-2\theta)}_{\leq 0} - \kappa \underbrace{\sin(2h - 4 \theta)}_{\geq 0} + \underbrace{(4-\kappa)}_{\leq 0} \underbrace{\vphantom{()}\sin 2h}_{\geq 0} \Bigr) \leq 0\, .
        \end{align*}
        Therefore, using the fact that $h$ is a solution to \eqref{eq:EL_h}, we obtain
        \begin{equation*}
            h''(\theta) \leq - \frac{\cos\theta}{\sin\theta} h'(\theta) \, 
        \end{equation*}
        for all $\theta \in [\theta_\kappa,\frac{\pi}{2}]$. Putting this into the remainder term, we get by partial integration
        \begin{align*}
            \frac{\pi}{4} - \delta_{\kappa} &\leq -h'(\tfrac{\pi}{2}) \delta_{\kappa} - \int_{\theta_\kappa}^{\frac{\pi}{2}} \frac{\cos\theta}{\sin\theta} h'(\theta) (\theta-\theta_\kappa) \, \mathrm{d}\theta \\
            &=  -h'(\tfrac{\pi}{2}) \delta_{\kappa} - \underbrace{ \frac{\cos\theta}{\sin\theta}(\theta-\theta_\kappa) h(\theta)\bigg|_{\theta_\kappa}^{\frac{\pi}{2}}}_{=0} + \int_{\theta_\kappa}^{\frac{\pi}{2}}  \left( - \frac{\theta-\theta_\kappa}{\sin^2\theta} + \frac{\cos\theta}{\sin\theta} \right) h(\theta)\mathrm{d}\theta \, .
        \end{align*}
        In the integral the first term is negative. Using the fact that $h \sin^{-1}\theta$ is bounded from above for $\theta \in [\theta_\kappa,\frac{\pi}{2}]$, we estimate further
        \begin{align*}
            \int_{\theta_\kappa}^{\frac{\pi}{2}}  \left( - \frac{\theta-\theta_\kappa}{\sin^2\theta} + \frac{\cos\theta}{\sin\theta} \right)h(\theta) \mathrm{d}\theta &\leq C \int_{\theta_\kappa}^{\frac{\pi}{2}} \cos\theta \, \mathrm{d} \theta \\
            &= C \int_{\theta_\kappa}^{\frac{\pi}{2}} \sin(\tfrac{\pi}{2}-\theta) \, \mathrm{d} \theta \leq C \delta_{\kappa}^2 \, .
        \end{align*}
        Taking the already established upper bound for $\delta_{\kappa}$, we find $\kappa_0\geq 4$ such that $C\delta_{\kappa} < 1$ for all $\kappa \geq \kappa_0$. Applying this result in the above inequality and rearranging yields, together with \cref{lemma:bound_derivative_h_2} and the fact that $\kappa \geq 4$, the estimates
        \begin{equation*}
            \frac{\pi}{4} \leq \left(- h'(\tfrac{\pi}{2}) + 2\right) \delta_{\kappa} \leq C \sqrt{\kappa} \delta_{\kappa} \, .
        \end{equation*}
        Dividing by $C\sqrt{\kappa}$, we arrive at the lower bound. 
        
        Putting this bound into \eqref{eq:bound_I3} and increasing $\kappa_0$ if necessary, we conclude for $\kappa \geq \kappa_0$ the estimate 
        \begin{equation} \label{eq:bound_I3_final}
            \kappa I_3 \geq \frac{\kappa}{4} \delta_{\kappa} \geq C\sqrt{\kappa}\, .
        \end{equation}
        Inserting all estimates \eqref{eq:bound_I1}, \eqref{eq:bound_I2}, and \eqref{eq:bound_I3_final} into \eqref{eq:bound_delta2E}, we arrive for $\kappa \geq \kappa_0$ at the estimate
        \begin{equation*}
            \delta^2 E[h](g) \leq \pi + \frac{C}{\sqrt[4]{\kappa}} - C\sqrt{\kappa} \, .
        \end{equation*}
        Increasing $\kappa_0$ if necessary, the right-hand side is indeed negative. 
    \end{proof}
\end{proposition}

\begin{proof}[Proof of \cref{theo:saddle_point}]
    We let $\kappa_0\geq 4$ be the constant from \cref{prop:second_variation_negative} and let $\kappa \geq \kappa_0$ be arbitrary but fixed. We set $\bm{m}_\kappa = \bm{m}$ to be the axisymmetric critical point of $\mathcal{E}$ with profile $h_\kappa = h$ obtained from the initial profile $h_{0,\kappa}$ and \cref{prop:no_blow_up}. We need to show that in any $H^1$-neighborhood of $\bm{m}$ there exist two elements where one has lower and the other has higher energy than $\bm{m}$. 
    
    The energy is continuous in $H^1$ and strictly decreasing along the flow as the initial map is not a solution of \eqref{eq:EL}. Therefore, we find for any $\tilde{\varepsilon} > 0$ an element $\bm{m}_+$ with $\|\bm{m}-\bm{m}_+ \|_{H^1} < \tilde{\varepsilon}$ and $\mathcal{E}(\bm{m}_+) > \mathcal{E}(\bm{m})$. Next, we set $\bm{m}_-$ to be the axisymmetric map with the profile $h_{\varepsilon} = h_\kappa + \varepsilon g_\kappa$, where $g_\kappa$ is the function defined in \eqref{eq:definition_g}. Then, according to \cref{prop:second_variation_negative}, there exists a constant $M>0$ such that due to \eqref{eq:energy_difference_hg_eps} we have
    \begin{equation*}
        E(h_{\varepsilon}) - E(h) \leq - M \varepsilon^2 + o(\varepsilon^2) < 0  \, 
    \end{equation*}
    for $\varepsilon$ small enough. This implies $\mathcal{E}(\bm{m}_-) < \mathcal{E}(\bm{m})$. Consequently, it remains to show that $\|\bm{m} - \bm{m}_-\|_{H^1} < \tilde{\varepsilon}$ for $\varepsilon$ small. Using trigonometric identities, we obtain
    \begin{equation*}
    \left|\bm{m}-\bm{m}_{-}\right|^{2} =\left|\left(\begin{array}{c}
    \cos \varphi\left(\sin h-\sin h_{\varepsilon}\right) \\
    \sin \varphi\left(\sin h-\sin h_{\varepsilon}\right) \\
    \cos h-\cos h_{\varepsilon}
    \end{array}\right)\right|^{2} \leq\left|h-h_{\varepsilon}\right|^{2} = \varepsilon^2 |g|^2\, ,
    \end{equation*}
    For the gradient norm we split it into
    \begin{equation*}
    \left|\nabla \bm{m}-\nabla \bm{m}_{-}\right|^{2}=\left|\partial_{\theta}\left(\bm{m}-\bm{m}_{-}\right)\right|^{2}+\frac{1}{\sin ^{2} \theta}\left|\partial_{\varphi}\left(\bm{m}-\bm{m}_{-}\right)\right|^{2}
    \end{equation*}
    and compute the terms individually:
    \begin{align*}
        \left|\partial_{\theta}\left(\bm{m}-\bm{m}_{-}\right)\right|^{2} &=\left|\left(\begin{array}{c}
    \cos \varphi\left(h^{\prime} \cos h-h_{\varepsilon}^{\prime} \cos h_{\varepsilon}\right) \\
    \sin \varphi\left(h^{\prime} \cos h-h_{\varepsilon}^{\prime} \cos h_{\varepsilon}\right) \\
    -h^{\prime} \sin h+h_{\varepsilon}^{\prime} \sin h_{\varepsilon}
    \end{array}\right)\right|^{2} \\
    &\leq\left|h^{\prime}-h_{\varepsilon}^{\prime}\right|^{2}+|h^{\prime} h_{\varepsilon}^{\prime}|\left|h-h_{\varepsilon}\right|^{2} \leq C\varepsilon^2  \, ,
    \end{align*}
    and
    \begin{equation*}
        \frac{1}{\sin ^{2} \theta} \left\lvert\, \partial_{\varphi}(\bm{m}-\bm{m}_{-})\right\rvert^{2} =\frac{1}{\sin ^{2} \theta}\left(\sin h-\sin h_{\varepsilon}\right)^{2} \leq \varepsilon^2 \frac{|g|^{2}}{\sin ^{2} \theta} = \varepsilon^2 |h'-1|^2 \leq C \varepsilon^2 \, ,
    \end{equation*}
    where we used that $h$ and $g$ are smooth and thus are bounded and have bounded derivatives. Combining these, we gain
    \begin{equation*}
        \left\|\bm{m}-\bm{m}_{-}\right\|_{H^{1}}^2  \leq C \varepsilon^2 < \tilde{\varepsilon}^2
    \end{equation*}
    for $\varepsilon$ small enough. Consequently, $\bm{m}_\infty$ is a saddle point of the functional $\mathcal{E}$.    
\end{proof}

\subsection{Saddle points of second type. Proof of Theorem 2}

The idea of the proof of \cref{theo:saddle_point_2} is twofold. In the first step, we repeat the arguments of the previous section and show that the critical points obtained in \cref{rema:hemispheric_wedge_condition_3} are saddle points for every $\kappa \geq 4$. In a second step, we depart from the critical point obtained in \cref{prop:kappa_4_solution} for $\kappa = 4$ and show that it is a saddle point of $\mathcal{E}$, and then by perturbation also for values close to 4.

\subsubsection{Saddle points for $\kappa > 4$}

We fix $\kappa > 4$ and we set $h = h_\kappa$ to be the smooth profile solution of \eqref{eq:EL_h} obtained from the initial profile $\theta \mapsto 2 \theta$ as in \cref{rema:hemispheric_wedge_condition_3}. Then $h$ satisfies the wedge condition \eqref{eq:hemispheric_wedge_condition_2}, too. We define the same map $g = g_\kappa\colon [0,\pi] \to \mathbb{R}$ as in \eqref{eq:definition_g} for the perturbation of the profile. Then, by \cref{lemma:bound_derivative_h_3}, we have $g(\theta) \geq 0$ for all $\theta \in [0,\pi]$. This observation suffices to show that the second variation of $\mathcal{E}$ at $h$ in direction $g$ is negative.

\begin{proposition}\label{prop:second_variation_negative_2}
    For all $\kappa \geq 4$ we have
    \begin{equation*}
        \delta^2 E_\kappa[h_\kappa](g_\kappa) < 0 \, .
    \end{equation*}
    \begin{proof}
        Since the computations leading to \eqref{eq:second_variation_g} do not depend on the details of the profile $h$, we obtain the same expression
        \begin{equation*}
            \delta^2 E[h](g) = 2 \int_0^{\frac{\pi}{2}} \lambda(\theta, h(\theta)) g \, \mathrm{d} \theta \, ,
        \end{equation*}
        with the function $\lambda$ defined as in \eqref{eq:lambda_function}. Using trigonometric identities, one shows that $\lambda$ is nonpositive on the whole set $W_2$ as long as $\kappa \geq 4$. Consequently, we have $\lambda(\theta, h(\theta)) \leq 0$ for all $\theta \in [0,\frac{\pi}{2}]$ by the wedge condition. Since $g \geq 0$, we find that the integrand is nonpositive. Moreover, it is also not constantly 0 as $h$ and $h'$ are not constant. 
    \end{proof}
\end{proposition}

\begin{cor} \label{cor:saddle_point_larger_4}
    For all $\kappa > 4$ the map $\bm{m}_\kappa$ defined by its profile $h_\kappa$ is a saddle point of $\mathcal{E}$.
    \begin{proof}
        The proof is analogous to the one of \cref{theo:saddle_point}.
    \end{proof} 
\end{cor}

Note that the heat flow argument to find a map with energy higher than the limit map fails for $\kappa = 4$ since the initial profile $\theta \mapsto 2\theta$ is already the stationary solution to the flow. 

\subsubsection{Saddle points in neighborhood of $\kappa = 4$} \label{section:saddle_point_kappa_4_nbh}

\begin{proposition} \label{prop:kappa_4_solution_saddle_point}
    For $\kappa = 4$, the axisymmetric map $\bm{m}$ defined by its profile $h(\theta) = 2\theta$ is a saddle point of $\mathcal{E}$.
    \begin{proof}
        By \cref{prop:kappa_4_solution}, it only remains to show the saddle point property. Again, it is sufficient to find test functions $g_1, g_2 \in E_{0,0}$ such that $\delta^2 E[h](g_1) < 0$ and $\delta^2 E[h](g_2) > 0$.

        Inserting the profile $h$ and the value $\kappa = 4$ into the second variation, we find by a trigonometric identity and integration by parts that
        \begin{equation*}
            \delta^2 E[h](g) = \int_0^\pi  \left\langle (-\hat{\mathcal{L}} - 4) g, g \right\rangle \sin\theta \, \mathrm{d} \theta \, ,
        \end{equation*}
        where the operator $\hat{\mathcal{L}}$ is defined by
        \begin{equation*}
            \hat{\mathcal{L}} =  \frac{1}{\sin\theta} \frac{\mathrm{d}}{\mathrm{d} \theta} \left( \sin\theta \frac{\mathrm{d}}{\mathrm{d} \theta}  \right) - \frac{1}{\sin^2\theta}\, .
        \end{equation*}
        This operator is the well-studied Sturm--Liouville operator of the general Legendre equation, see Appendix~\ref{section:sturm_liouville_legendre}. Accordingly, the eigenvalues of the shifted operator $\mathcal{L} = -\hat{\mathcal{L}} -4$ are given by $\lambda_l = l(l+1)-4$ with the same eigenfunctions as $\hat{\mathcal{L}}$. In particular, the first eigenvalue is $\lambda_1 = -2$ with eigenfunction $P_1^1(\cos\theta) = \sin \theta$ and the second eigenvalue is $\lambda_2 = 2$ with eigenfunction $P_2^1(\cos\theta) = \sin 2\theta$. Both these eigenfunctions are elements of $E_{0,0}$ and hence, $\bm{m}$ is a saddle point of $\mathcal{E}$.
    \end{proof}  
\end{proposition}

The operator $\hat{\mathcal{L}}$ will play a key role in the following. By standard theory on self-adjoint operators, we know that its maximal domain $D(\hat{\mathcal{L}})$, defined in \eqref{eq:domain_legendre_operator}, endowed with the graph norm
\begin{equation*}
    \|f\|_D = \|f\|_{L_{\sin}^2} + \|\hat{\mathcal{L}} f \|_{L_{\sin}^2} \, ,
\end{equation*}
is a Banach space.

\begin{lemma} \label{prop:kappa_4_perturbation_sol}
    There exists a neighborhood $I \subset \mathbb{R}_{>0}$ of $4$ such that for all $\kappa \in I$ there exists a map $h_\kappa \in E_{0,2}$ that is a solution to \eqref{eq:EL_h} for that respective $\kappa$. The map $\kappa \mapsto h_\kappa$ is continuously differentiable as map $(I, |\cdot|) \to (E_{0,2}, \| \cdot \|_D)$. 
    \begin{proof}
        To use the implicit function theorem, we first rewrite \eqref{eq:EL_h} as an equation for the function $\eta \in D(\hat{\mathcal{L}})$ using $h_\eta = 2\theta + \eta$:
        \begin{align*}
            &\phantom{{}={}}h_\eta'' + \frac{\cos\theta}{\sin\theta} h_\eta' - \frac{\sin 2h_\eta}{2\sin^2\theta} - \frac{\kappa}{2} \sin(2h_\eta-2\theta) \\
            &= \eta'' + \frac{\cos\theta}{\sin\theta} \eta' + 2 \frac{\cos\theta}{\sin\theta} - \left( 4\frac{\cos\theta}{\sin\theta} + \left(\kappa - 4\right) \sin2\theta \right) \frac{\cos 2\eta}{2} \\
            &\phantom{{}={}}+ \left(4 - \frac{1}{\sin^2\theta} - (\kappa-4)\cos 4\theta \right) \frac{\sin2\eta}{2} \eqqcolon T(\eta,\kappa) \, , \addtocounter{equation}{1}\tag{\theequation} \label{eq:definition_T}
        \end{align*}
        where $T\colon D(\hat{\mathcal{L}}) \times \mathbb{R} \to L_{\sin}^2([0,\pi])$ is the operator to which we will apply the implicit function theorem. By construction and \cref{prop:kappa_4_solution}, we know that $T(0,4) = 0$. We compute the partial derivative in direction $g \in D(\hat{\mathcal{L}})$
        \begin{align*}
            \frac{\mathrm{d}}{\mathrm{d} t} T(\eta+tg, \kappa) \Big|_{t=0} &= g'' + \frac{\cos\theta}{\sin\theta} g' + \left( 4\frac{\cos\theta}{\sin\theta} + \left(\kappa - 4\right) \sin2\theta \right) \sin 2\eta \; g\\
            &\phantom{{}={}}+ \left(4 - \frac{1}{\sin^2\theta} - (\kappa-4) \cos4\theta \right) \cos2\eta \; g \\
            & \eqcolon \mathcal{L}(\eta,\kappa) g \, . \addtocounter{equation}{1}\tag{\theequation} \label{eq:definition_Lcal}
        \end{align*}
        We see that this derivative is a linear operator and that $\mathcal{L}(0,4)$ has the form
        \begin{equation*}
            \mathcal{L}(0,4) = \frac{1}{\sin\theta} \frac{\mathrm{d}}{\mathrm{d} \theta} \left( \sin\theta \frac{\mathrm{d}}{\mathrm{d} \theta}  \right) - \frac{1}{\sin^2\theta}  + 4 = \hat{\mathcal{L}} + 4 \, .
        \end{equation*}
        Since we know that $-4$ is in the resolvent set of $\hat{\mathcal{L}}$, we deduce that $\mathcal{L}(0,4)$ is a Banach space isomorphism. We claim that $T$ is continuously differentiable. As $T$ is linear in $\kappa$, it suffices to show that the partial derivative in $\eta$-direction, $\mathcal{L}(\eta,\kappa)$, is a bounded operator and that it is continuous as a map $(\eta,\kappa) \mapsto \mathcal{L}(\eta,\kappa)$. Starting from the definition of $\mathcal{L}(\eta,\kappa)$, we see that we can express it as
        \begin{align*}
            \mathcal{L}(\eta,\kappa) &= \hat{\mathcal{L}} + \left( 4\frac{\cos\theta}{\sin\theta} + \left(\kappa - 4\right) \sin2\theta \right) \sin 2\eta + \bigl(4 - (\kappa-4) \cos4\theta \bigr) \cos2\eta\\
            &\phantom{{}={}}+ \frac{1}{\sin^2\theta} \left( 1- \cos 2\eta \right) \eqqcolon \hat{\mathcal{L}} + f(\eta,\kappa) \, , \addtocounter{equation}{1}\tag{\theequation} \label{eq:def_Lcal}
        \end{align*}
        where $f(\eta,\kappa)$ is a scalar function in $L_{\sin}^2$ according to \cref{lemma:sl_better_Lp_bound} for the singular terms and by the boundedness of the non-singular terms. Consequently, for $g \in D(\hat{\mathcal{L}})$ we compute, applying \cref{lemma:sl_linf_bound},
        \begin{equation*}
            \|\mathcal{L}(\eta,\kappa)g \|_{L_{\sin}^2} \leq \|\hat{\mathcal{L}}g \|_{L_{\sin}^2} + \|f (\eta,\kappa)\|_{{L_{\sin}^2}} \|g \|_{L^\infty} \leq (1 + \|f (\eta,\kappa)\|_{L_{\sin}^2}) \|g \|_D \, .
        \end{equation*}
        Thus, we find that $\mathcal{L}(\eta,\kappa)$ is bounded. To show continuity, we use the uniform continuity of the trigonometric functions and invoke \cref{lemma:sl_better_Lp_bound,lemma:sl_linf_bound}. This yields a bound of the form
        \begin{align*}
            &\phantom{{}={}}\|\mathcal{L}(\eta_1, \kappa_1) - \mathcal{L}(\eta_2,\kappa_2) \| \leq \|f(\eta_1, \kappa_1) - f(\eta_2, \kappa_2) \|_{L_{\sin}^2} \\
            &\leq C \left( \left\| \frac{\eta_1-\eta_2}{\sin} \right\|_{L_{\sin}^2} + \left\| \frac{\eta_1-\eta_2}{\sin} \right\|_{L_{\sin}^4}^2 + (1+\kappa_2) \| \eta_1-\eta_2\|_{L^\infty} + |\kappa_1 - \kappa_2| \right) \\
            &\leq C \left( (1+\kappa_2) \| \eta_1-\eta_2 \|_{D} + \|\eta_1-\eta_2 \|_{D}^2 + |\kappa_1 - \kappa_2| \right) \, . \addtocounter{equation}{1}\tag{\theequation} \label{eq:Lcal_cont}
        \end{align*}
        Thus, the map $(\eta,\kappa)\mapsto \mathcal{L}(\eta,\kappa)$ is continuous. 
        
        As a result, we invoke the implicit function theorem (c.f.~\cite[Thm. 3.5.4]{buffoni2016}) and find neighborhoods $U \subset D(\mathcal{L})$ of $0$ and $I \subset \mathbb{R}$ of $4$, and a function $\eta\colon I \to U$, which is continuously differentiable, such that $\eta_4 = 0$ and $T(\eta_\kappa,\kappa) = 0$ for all $\kappa \in I$. Setting $h_\kappa = 2\theta + \eta_\kappa$, we find that $h_\kappa$ is a solution to \eqref{eq:EL_h} for all $\kappa \in I$. Moreover, by \cref{prop:legendre_operator}, we deduce $h_\kappa \in E_{0,2}$. Finally, since for $\kappa = 0$ there are no solutions to \eqref{eq:EL_h} in $E_{0,2}$ as the only harmonic map with mapping degree $0$ is the constant map, we assert that $I \subset \mathbb{R}_{>0}$. 
    \end{proof}
\end{lemma}

\begin{proposition} \label{prop:saddle_point_kappa_4_nbh}
    There exist $0 < \kappa_1 < 4 < \kappa_2$ with $(\kappa_1, \kappa_2)\subset I$ such that for all $\kappa \in (\kappa_1,\kappa_2)$ the axisymmetric map $\bm{m}_\kappa$ with profile $h_\kappa$ provided by \cref{prop:kappa_4_perturbation_sol} is a saddle point of the energy functional $\mathcal{E}$.
    \begin{proof}
        Again, it suffices to find test functions $g_1,g_2\in E_{0,0}$ such that $\delta^2 E[h_\kappa](g_1) < 0$ and $\delta^2 E[h_\kappa](g_2) > 0$. For $h_\kappa = 2\theta + \eta_\kappa$ and $g\in D(\hat{\mathcal{L}})$ we write, using a calculation similar to the one in \eqref{eq:definition_Lcal},   
        \begin{equation*}
            \delta^2 E[h_\kappa](g) = \int_0^\pi \left\langle -\mathcal{L}(\eta_\kappa,\kappa) g,g \right\rangle \sin\theta \, \mathrm{d} \theta \, .
        \end{equation*}
        The map $(\eta,\kappa) \mapsto \mathcal{L}(\eta,\kappa)$ is continuously differentiable with respect to the operator norm induced by the norm on $D(\hat{\mathcal{L}})\times \mathbb{R}$. This follows from the fact that both partial derivatives are continuous. For the $\kappa$-direction this is due to $\mathcal{L}$ being linear in $\kappa$. Existence and continuity of the directional derivative in the $\eta$-direction can be shown as in \eqref{eq:Lcal_cont}, using the chain rule and the differentiability of $\cos$ and $\sin$, as well as invoking \cref{lemma:sl_better_Lp_bound,lemma:sl_linf_bound}. 
        
        Combining this observation with the continuous differentiability of the map $\kappa \mapsto \eta_\kappa$, we deduce by the chain rule that the map $\kappa \mapsto \mathcal{L}(\eta_\kappa,\kappa)$ is continuously differentiable. Consequently, by standard theory on perturbation of simple eigenvalues (c.f.~\cite[Prop. 3.6.1]{buffoni2016}), recalling the eigenvalues of the unperturbed operator $-\mathcal{L}(0,4)$, we find that there exist $\kappa_1 < 4 < \kappa_2$ such that for all $\kappa \in (\kappa_1,\kappa_2)$ the first eigenvalue $\lambda_1(\kappa)$ of $-\mathcal{L}(\eta_\kappa,\kappa)$ with eigenfunction $g_1(\kappa)$ is still negative and the second eigenvalue $\lambda_2(\kappa)$ with eigenfunction $g_2(\kappa)$ is still positive. Consequently, $\bm{m}_\kappa$ is a saddle point of $\mathcal{E}$ for $\kappa \in (\kappa_1,\kappa_2)$.
    \end{proof}
\end{proposition}

\begin{remark}
    For $\kappa \in (4,\kappa_2)$ we have now two ways to obtain a critical point of the energy functional $\mathcal{E}$. It is an interesting question if these two methods yield the same critical points. Moreover, it is not guaranteed that the critical points obtained via the implicit function theorem still obey the hemispheric symmetry. Our conjecture is that for every $\kappa \geq 4$ there only exists one critical point of the energy functional $\mathcal{E}$ with profiles in $H_{0,2}$ and thus the two methods yield the same critical points. This is backed by numerical evidence.
\end{remark}

\subsubsection{Stitching together both regimes}

\begin{proof}[Proof of \cref{theo:saddle_point_2}]
    We let $\kappa_1$ be the one from \cref{prop:kappa_4_perturbation_sol}. Then for $\kappa \in (\kappa_1,4]$ we set $h_\kappa$ to be the profile from \cref{prop:saddle_point_kappa_4_nbh}. Further, for $\kappa > 4$ we set $h_\kappa$ to be the one from \cref{cor:saddle_point_larger_4}. We denote the corresponding maps by $\bm{m}_\kappa$. We assert that $Q(\bm{m}_\kappa) = 0$ as $h_\kappa \in E_{0,2}$. Moreover, by \cref{prop:saddle_point_kappa_4_nbh} and \cref{cor:saddle_point_larger_4}, we find that $\bm{m}_\kappa$ is a saddle point of the energy functional $\mathcal{E}$ in the claimed sense. The saddle point is different from the one obtained in \cref{theo:saddle_point} since the profiles in the latter are in $H_{0,0}$ and the ones in this theorem are in $E_{0,2}$. 
\end{proof}

\section*{Acknowledgements}

DM thanks the Department of Mathematics at the University of British Columbia for its hospitality. Moreover, he thanks Maria G.\ Westdickenberg for constant support and helpful suggestions.

This work was funded by the DFG-Graduiertenkolleg \emph{Energy, Entropy, and Dissipative Dynamics (EDDy)}, project no. 320021702/GRK2326.

\appendix

\section{Variant of maximum principle}

Recall the definition of the parabolic cylinder for $U\subset \mathbb{R}^n$ and $T>0$ to be $U_T = U \times (0 ,T]$. 
\begin{lemma} \label{lemma:max_principle_variant}
    Let $I\coloneqq (a,b) \subset \mathbb{R}$ be an open bounded interval and assume $u \in C^2(I_T)\cup C(\overline{I}_T)$ satisfies
    \begin{align*}
        &0 \leq u_t - u'' + d(u,x) u' + c(u,x) u \, \\
        &u|_{t=0} = u_0 \\
        &u(a,t) = g(t), \; u(b,t) = h(t) \, ,
    \end{align*}
    where $d$ and $c$ are sufficiently smooth functions $\mathbb{R}\times\mathbb{R} \to \mathbb{R}$ and $g,h\colon [0,T] \to [0,\infty)$ are continuous functions satisfying the compatibility conditions $g(0) = u_0(a)$ and $h(0) = u_0(b)$. Additionally, assume there exists $\delta>0$ such that $c(u(x, t), x) > - \delta$ holds for all $(x,t) \in I_T$. Then, $u_0 \geq 0$ on $\overline{I}$ implies $u\geq 0$ on $I_T$.
    \begin{proof}
        We define 
        \begin{equation*}
            v(x,t) = u(x,t) e^{-\delta t} \, .
        \end{equation*}
         Then $v$ satisfies
        \begin{align*}
            &0 \leq v_t - v'' + d(u,x) v' + (c(u,x)+\delta) v \, \\
            &v|_{t=0} = u_0 \\
            &v(a,t) \geq 0, \; v(b,t) \geq 0
        \end{align*}
        and the statement follows from the usual maximum principle for parabolic equations, e.g.~\cite[7.1.4. Thm.\ 9]{evans2010}.
    \end{proof}
\end{lemma}

\section{Sturm--Liouville theory: The Legendre operator} \label{section:sturm_liouville_legendre}

We consider the Sturm--Liouville problem of the associated Legendre equation for the value $m=1$ in the variable $x = \cos\theta$ on the interval $(0,\pi)$. That is, we study the equation
\begin{equation*}
    \frac{1}{\sin\theta} \frac{\mathrm{d}}{\mathrm{d} \theta} \left( \sin\theta \frac{\mathrm{d}}{\mathrm{d} \theta}  h \right) - \frac{m^2}{\sin^2\theta} h = 0 \, 
\end{equation*}
for $m=1$. This equation is well-known as the the azimuthal angle part of the Laplace equation after using separation of variables in spherical coordinates. We define the Legendre operator $\hat{\mathcal{L}}$ as
\begin{equation*}
    \hat{\mathcal{L}} = \frac{1}{\sin\theta} \frac{\mathrm{d}}{\mathrm{d} \theta}  \sin\theta \frac{\mathrm{d}}{\mathrm{d}\theta}  - \frac{1}{\sin^2\theta} \, .
\end{equation*}
In this section we briefly collect some properties of this operator resulting from the Sturm--Liouville theory. Further details can be found in~\cite{zettl2005,arfken2013}. A complete set of eigenfunctions of $\hat{\mathcal{L}}$ is given by the associated Legendre functions 
\begin{equation*}
    \Psi_l(\theta) = P_l^1(\cos\theta) = \frac{1}{2^l l!}\sin\theta \frac{\mathrm{d}^{l+1} }{\mathrm{d} x^{l+1}} (x^2-1)^l \big|_{x=\cos\theta}  \, 
\end{equation*}
for $l \in \mathbb{N}$. The corresponding discrete eigenvalues are simple and given by
\begin{equation*}
    \hat{\mathcal{L}}\Psi_l = -l(l+1) \Psi_l \, .
\end{equation*} 
Moreover, the eigenfunctions form an orthogonal basis of the Hilbert space $L_{\sin}^2((0,\pi))$, which is defined as the space of square integrable functions with respect to the measure $\sin\theta \mathrm{d} \theta$ on the interval $(0,\pi)$. 

The maximal domain $D(\hat{\mathcal{L}})$ is given by
\begin{equation} \label{eq:domain_legendre_operator}
    D(\hat{\mathcal{L}}):=\big \{f \in L_{\sin}^2((0,\pi)) : f, f'\sin \in \text{AC}_{\text{loc}}((0,\pi)), \quad \hat{\mathcal{L}}f\in L_{\sin}^2((0,\pi)) \big \} \, ,
\end{equation}
where $\text{AC}_{\text{loc}}((0,\pi))$ is the space of locally absolutely continuous functions~\cite[c.f.\ Part 4]{zettl2005}. The operator $\hat{\mathcal{L}}$ is self-adjoint on $D(\hat{\mathcal{L}})$ with respect to the inner product $\left\langle \cdot,\cdot \right\rangle _{L_{\sin}^2}$.

\begin{proposition} \label{prop:legendre_operator}
    Let $f\in D(\hat{\mathcal{L}})$. Then,
    \begin{equation*}
        f(0) = f(\pi) = 0 \quad\text{~and~}\quad \lim_{\theta\to 0} f'(\theta) \sin\theta = \lim_{\theta\to \pi} f'(\theta) \sin\theta = 0 \, .
    \end{equation*}
    \begin{proof}
        The Lagrange identity for the operator $\hat{\mathcal{L}}$ and $f,g\in D(\hat{\mathcal{L}})$ reads
        \begin{equation} \label{eq:lagrange_identity}
            g\hat{\mathcal{L}}f  - f\hat{\mathcal{L}}g  = \frac{1}{\sin\theta} \frac{\mathrm{d}}{\mathrm{d} \theta} \left( \sin\theta (f'g - fg')\right) \, .
        \end{equation}
        Setting $g = \sin$ and integrating, we find, using the fact that $g$ is the first eigenfunction of $\hat{\mathcal{L}}$, that
        \begin{equation*}
            \int_0^\pi (\hat{\mathcal{L}} - 2)f \, \sin^2\theta \, \mathrm{d} \theta = f' \sin^2\theta - f\sin\theta\cos\theta \Big|_0^\pi\, .
        \end{equation*}
        Now, since $f,g \in D(\hat{\mathcal{L}})$, we see that by the Cauchy--Schwarz inequality the left-hand side exists and thus, the limits on the right-hand side exist as well. We denote the limits by $\Phi_0(f)$ and $\Phi_\pi(f)$, respectively. We integrate the Lagrange identity again, but now from $0$ to $\theta \in (0,\pi)$, which yields
        \begin{equation} \label{eq:legendre_identity2}
            \int_0^\theta (\hat{\mathcal{L}} - 2)f \, \sin^2\tau\, \mathrm{d} \tau = f'\sin^2\theta - f\sin\theta\cos\theta - \Phi_0(f)\, .
        \end{equation}
        Dividing by $\sin^3\theta$ and rearranging gives us
        \begin{equation}\label{eq:legendre_identity1}
            \frac{\mathrm{d}}{\mathrm{d} \theta} \left( \frac{1}{\sin\theta} f \right) = \frac{\Phi_0(f)}{\sin^3\theta} + \frac{B(\theta)}{\sin\theta} \, ,
        \end{equation}
        where 
        \begin{equation*}
            B(\theta) = \frac{1}{\sin^2\theta} \int_0^\theta (\hat{\mathcal{L}} - 2)f \sin^2\tau \, \mathrm{d} \tau \, .
        \end{equation*}
        We find the following bound for $B(\theta)$, using Hölder's inequality:
        \begin{align*}
            |B(\theta)| &\leq \frac{1}{\sin^2\theta} \left( \|\hat{\mathcal{L}}f \|_{\mathrlap{L_{\sin}^2}} \quad + 2\|f \|_{L_{\sin}^2} \right) \left( \int_0^\theta \sin^3 \tau \, \mathrm{d}\tau \right)^{\;\;\mathclap{1/2}} \;\leq C \left( \|\hat{\mathcal{L}}f \|_{\mathrlap{L_{\sin}^2}} \quad + 2\|f \|_{L_{\sin}^2} \right) \, ,
        \end{align*}
        where $C>0$ is global. Integrating \eqref{eq:legendre_identity1}, now from $t \in (0,\tfrac{\pi}{2})$ to $\tfrac{\pi}{2}$, gives us
        \begin{equation*}
            f(t) = \left( f(\tfrac{\pi}{2}) - \int_t^\frac{\pi}{2} \frac{\Phi_0(f)}{\sin^3\theta} \, \mathrm{d}\theta- \int_t^\frac{\pi}{2} \frac{B(\theta)}{\sin\theta} \, \mathrm{d}\theta\right)\sin t \, .
        \end{equation*}
        The first integral can be evaluated to
        \begin{equation*}
            \int_t^\frac{\pi}{2} \frac{\Phi_0(f)}{\sin^3\theta} \, \mathrm{d}\theta = \frac{\Phi_0(f)}{2} \left( \ln \tan \left( \frac{t}{2} \right) - \frac{\cos t}{\sin^2t}\right)\, ,
        \end{equation*}
        whereas for the second integral we find, using the antiderivative of $\sin^{-1}$, the bound
        \begin{equation} \label{eq:legendre_bound_B}
            \Bigl| \int_t^\frac{\pi}{2} \frac{B(\theta)}{\sin\theta} \, \mathrm{d}\theta \Bigr| \leq C \left( \|\hat{\mathcal{L}}f \|_{L_{\sin}^2} + 2\|f \|_{L_{\sin}^2} \right) |\ln \tan \left( \frac{t}{2} \right)| \, .
        \end{equation}
        Since 
        \begin{equation*}
            \ln \tan \left( \frac{t}{2} \right) \sin t \longrightarrow 0 \, , \, \text{~as~} \,  t \to 0 \, ,
        \end{equation*}
        we see that we can write $f(t)$ as
        \begin{equation*}
            f(t) = G(t) - \frac{\Phi_0}{2} \frac{\cos t}{\sin t}\, ,
        \end{equation*}
        where $G$ is a bounded function on $[0, \tfrac{\pi}{2}]$. Now, since $f \in L_{\sin}^2((0,\pi))$, we deduce that $\Phi_0(f) = 0$ must hold as $x\mapsto \cos x \sin^{-1}x$ is not square integrable on $(0,\pi)$. Thus, we write 
        \begin{equation} \label{eq:legendre_expression_f}
            f(t) = f(\tfrac{\pi}{2}) \sin t - \sin t \int_t^\frac{\pi}{2} \frac{B(\theta)}{\sin\theta} \, \mathrm{d}\theta\, ,
        \end{equation}
        from which follows, using \eqref{eq:legendre_bound_B},
        \begin{equation*}
            f(t) \longrightarrow 0 \, , \, \text{~as~} \,  t\to 0 \, .
        \end{equation*}
        Going back to \eqref{eq:legendre_identity2}, we assert that
        \begin{equation*}
             f'(t)\sin t =  \, f(t)\cos t +  B(t)\sin t\, .
        \end{equation*}
        Recalling that $B$ is bounded, we find that
        \begin{equation*}
            \, f'(t)\sin t  \longrightarrow 0 \, , \, \text{~as~} \,  t\to 0 \, .
        \end{equation*}
        We repeat the same arguments for $t\to \pi$ by integrating \eqref{eq:legendre_identity1} from $\tfrac{\pi}{2}$ to $t$ to conclude that both $f(t) \to 0$ and $f'(t)\sin t  \to 0$ as $t \to \pi$.
    \end{proof}
\end{proposition}

The proof also delivers the estimate for the following two Sobolev-type inequalities.
\begin{lemma} \label{lemma:sl_better_Lp_bound}
    For any $p \geq 2$ and $f\in D(\hat{\mathcal{L}})$ we have
    \begin{equation*}
        \Bigl\| \frac{f}{\sin} \Bigr\|_{L_{\sin}^p} \leq C \|f \|_D \, ,
    \end{equation*}
    where the constant $C>0$ depends only on $p$.
    \begin{proof}
        For $\theta\in (0,\tfrac{\pi}{2})$, by \eqref{eq:legendre_expression_f} and the estimate in \eqref{eq:legendre_bound_B} we find that
        \begin{equation} \label{eq:legendre_bound_fsin}
            \Bigl| \frac{f(\theta)}{\sin(\theta)} \Bigr|^p \leq C \|f\|_D^p \, |\ln \tan \left( \frac{\theta}{2} \right)|^p \, .
        \end{equation}
        For $\theta \in (\tfrac{\pi}{2},\pi)$ we get a similar expression. The function
        \begin{equation*}
            \theta\mapsto \Bigl| \ln \tan \left( \frac{\theta}{2} \right) \Bigr|^p \sin\theta \, 
        \end{equation*}
        is bounded on $[0,\pi]$ and thus integrable, and therefore the claim follows with \eqref{eq:legendre_bound_fsin}.
    \end{proof}    
\end{lemma}

\begin{lemma} \label{lemma:sl_linf_bound}
    Let $f\in D(\hat{\mathcal{L}})$. Then there is a global constant $C>0$ such that
    \begin{equation*}
        \|f \|_{L^\infty} \leq C \|f \|_D \, .
    \end{equation*}
    \begin{proof}
        By \cref{prop:legendre_operator}, the function $f f' \sin$ vanishes at the endpoints of the interval $(0,\pi)$. Hence, by partial integration we find that
        \begin{align*}
            \left\langle f,-\hat{\mathcal{L}}f\right\rangle =  \int_0^\pi f (-\hat{\mathcal{L}}f) \, \sin \theta \, \mathrm{d} \theta \
            & = \int_0^\pi f'^2 \sin\theta \, \mathrm{d} \theta + \int_0^\pi \frac{f^2}{\sin\theta} \, \mathrm{d} \theta \, ,
        \end{align*}
        from which we conclude by Young's inequality that
        \begin{equation*}
            \|f' \|_{L_{\sin}^2}^2 + \| \frac{f}{\sin} \|_{L_{\sin}^2}^2 \leq \|f \|_{L_{\sin}^2}^2 + \|\hat{\mathcal{L}}f \|_{L_{\sin}^2}^2 = \|f\|_D^2   \, .
        \end{equation*}
        This implies with $f(0) = 0$ that for $\theta \in (0,\pi)$ we have
        \begin{equation*}
            \frac{1}{2} f(\theta)^2 = \int_0^\theta f f' \, \mathrm{d} \tau \leq \|f' \|_{L_{\sin}^2}^2 + \| \frac{f}{\sin} \|_{L_{\sin}^2}^2 \leq \|f \|_D^2\, .
        \end{equation*}
        The statement then follows immediately.
    \end{proof}
\end{lemma}

\bibliographystyle{amsplain}
\bibliography{bib_books.bib, bib_papers_sphere.bib}

\end{document}